\documentclass[12pt]{amsart}
\usepackage{amsmath,amssymb,amsthm,bm,longtable,color}
 \usepackage{a4wide}

\newtheorem{thm}{Theorem}[section]
\newtheorem*{thm*}{Theorem}
\newtheorem{conj}[thm]{Conjecture}
\newtheorem{cor}[thm]{Corollary}
\newtheorem*{conj*}{Conjecture}
\newtheorem{lem}[thm]{Lemma}
\newtheorem{prop}[thm]{Proposition}

\newtheorem{qn}[thm]{Question}

\theoremstyle{remark}
\newtheorem{rem}[thm]{Remark}

\theoremstyle{definition}
\newtheorem{defn}[thm]{Definition}
\newcounter{claim}[thm]

\newcommand{\soc}{\mathrm{soc}}
\newcommand{\PSL}{\mathrm{PSL}}

\newcommand{\sym}{\mathrm{Sym}}
\newcommand{\alt}{\mathrm{Alt}}

\newcommand{\sh}{\mathrm{Sh}}
\newcommand{\GL}{\mathrm{GL}}
\newcommand{\AGL}{\mathrm{AGL}}
\newcommand{\supp}{\mathrm{supp}}
\newcommand{\PSU}{\textnormal{PSU}}
\newcommand{\Ree}{\textnormal{Ree}}
\newcommand{\Sz}{\textnormal{Sz}}
\newcommand{\aut}{\textnormal{Aut}}

\newcommand{\Sp}{\textnormal{Sp}}

\title{Generalised shuffle groups}

\author{Carmen Amarra, Luke Morgan, Cheryl~E.~Praeger}

\thanks{The first author was supported by a Post Doctoral Research Award (FRASDP) of the University of the Philippines. The second author acknowledges the Australian Research Council Grant DE160100081 and the  Slovenian Research Agency (research program P1-0285).}

\address{Carmen Amarra \\ Institute of Mathematics, University of the Philippines Diliman, C. P. Garcia Avenue, Diliman, Quezon City 1101, Philippines }
\email{mcamarra@math.upd.edu.ph}

\address{Luke Morgan \\ University of Primorska, UP FAMNIT, Glagolja\v{s}ka 8, 6000 Koper, Slovenia, and University of Primorska, UP IAM, Muzejski trg 2,  6000 Koper, Slovenia.}
\email{luke.morgan@famnit.upr.si}

\address{Cheryl E.~Praeger \\ Centre for the Mathematics of Symmetry and Computation, Department of Mathematics and Statistics, The University of Western Australia, 35 Stirling Highway, Crawley, 6009, Western Australia, Australia}
\email{cheryl.praeger@uwa.edu.au}


\begin{document}

\begin{abstract}
The mathematics of shuffling a deck of $2n$ cards with two ``perfect shuffles'' was brought into clarity by Diaconis, Graham and Kantor. Here we consider a generalisation of this problem, with a so-called ``many handed dealer'' shuffling $kn$ cards by cutting into $k$ piles with $n$ cards in each pile and using $k!$ shuffles. A conjecture of Medvedoff and Morrison suggests that all possible permutations of the deck of cards are achieved, so long as $k\neq 4$ and $n$ is not a power of $k$.  We confirm this conjecture for three doubly infinite families of integers: all $(k,n)$ with $k>n$; all $(k, n)\in \{ (\ell^e, \ell^f )\mid \ell \geqslant 2, \ell^e>4, f \ \mbox{not a multiple of}\ e\}$; and all $(k,n)$ with 
$k=2^e\geqslant 4$ and $n$ not a power of $2$. We open up a more general study of shuffle groups, which admit an arbitrary subgroup of shuffles.
\end{abstract}

\maketitle

\section{Introduction}

The crux of a card trick performed with a deck of cards usually depends on understanding how shuffles of the deck  change the order of the cards. By understanding which permutations are possible, one knows if a given card may be brought into a certain position. The two standard ``perfect'' shuffles of a deck of 52 cards involve ``cutting'' the deck into two piles, holding one pile in each hand, and then perfectly interlacing the two piles. Depending on whether the card from the left or right hand pile ends up on top, these two shuffles are referred to as the in-shuffle and out-shuffle respectively. 


To understand fully what permutations can be achieved on a deck of cards via unlimited repetitions of these two operations, we have to know the permutation group generated by the two shuffles. If we have a set of $2n$ cards (for an integer $n$) labelled with integers, and write $\sigma$ and $\delta$ for the permutations induced by the in- and out-shuffle, respectively, we are asking about the structure of the subgroup $\langle \sigma ,\delta \rangle$  of $\sym(2n)$. Diaconis, Graham and Kantor \cite{DGK} were the first to answer this problem completely -- although there were partial results earlier for certain values of $n$ (see the discussion in \cite[\S3]{DGK}).
They observed first that the group $\langle \sigma , \delta \rangle$ preserves ``central symmetry'', that is, the group preserves a partition of the $2n$ cards into $n$ parts each of size two, and so $\langle \sigma , \delta \rangle$ is a subgroup of the Weyl group of type $B_n$, which we denote simply by $B_n$; the group $B_n\cong C_2 \wr \sym(n)$ (an example of a wreath product in imprimitive action).  For a precise description of $\langle \sigma, \delta \rangle$  as a subgroup of $B_n$, we require some definitions.




The derived subgroup of $B_n$ has index four, and the quotient is elementary abelian. Thus there are three index two subgroups, which are the kernels of sign-like homomorphisms. If $g\in B_n$, $\mathrm{sgn}(g)$ is the sign of $g$ as a permutation of $2n$ cards and $\overline{\mathrm{sgn}}(g) $ is the sign of $g$ as a permutation of the $n$ parts of size two. Both maps are homomorphisms from $B_n$ to $\{-1,1\}$, so there is a third map which is their product $\mathrm{sgn}\overline{\mathrm{sgn}}(g):=\mathrm{sgn}(g)\overline{\mathrm{sgn}}(g)$.
With these definitions, the structure found by Diaconis et al.~ is as follows.

\begin{thm} \cite[Theorem 1]{DGK} \label{dgk}
The structure of the  shuffle group $\langle \sigma, \delta \rangle$ on $2n$ points, where $n\geqslant2$, is given in Table~\ref{tab:the shuffle group}.
\end{thm}

\begin{table}[h]
\label{tab:the shuffle group}
\begin{tabular}{ l | l }
Size of each pile $n$ & Shuffle group $\langle \sigma , \delta \rangle$ \\
\hline 
$n=2^f$ for some positive integer $f$ &  $C_2 \wr C_{f+1}$ \\ 
$n\equiv 0 \pmod{4}$, $n\geqslant 20$ and $n$ is not a power of $2$ &  $\ker(\mathrm{sgn}) \cap \ker(\overline{\mathrm{sgn}})$  \\ 
$n\equiv 1 \pmod{4}$ and $n\geqslant 5 $ &  $ \ker(\overline{\mathrm{sgn}})$  \\ 
$n\equiv 2 \pmod{4}$ and $n\geqslant 10$  &  $B_n$  \\ 
$n\equiv 3 \pmod{4}$ &  $\ker(\mathrm{sgn}\overline{\mathrm{sgn}})$  \\ 
$n=6$ &  $C_2^6 \rtimes \mathrm{PGL}(2,5)$  \\ 
$n=12$ &  $C_2^{11} \rtimes   M_{12}$\\ \hline
\end{tabular}
\smallskip
\caption{The   shuffle group on $2n$ points}
\end{table}

The appearance of Mathieu's group $M_{12}$ in the  list is quite remarkable. In the case where $n=2^f$, the wreath product  is acting in \emph{product action} on $2^{f+1}$ points (see Section~\ref{sub:perm}). All the groups in Table~\ref{tab:the shuffle group} are imprimitive subgroups of $\sym(2n)$ preserving a partition with $n$ parts of size $2$.


A natural question concerning the mathematics of shuffling is raised at the end of the paper of Diaconis et ~al. We could divide a pack of $kn$ cards into $k$ piles (with the first $n$ cards in the first pile, the second $n$ cards in the second pile, and so on), and then perform a shuffle. There are now more shuffles to consider -- because there are more permutations of the piles of cards. We again denote by $\sigma$ the standard shuffle, where the cards are picked up from left to right without permuting the piles at all. For a permutation $\tau$ of the piles, we have an induced permutation $\rho_\tau \in \sym(kn)$ of the set of cards. The corresponding shuffle obtained by first permuting the piles according to $\tau$ and then picking up the cards is therefore equal to $\rho_\tau \sigma$, and $\sigma = \rho_\tau\sigma$ with $\tau$ the identity element of $\sym(k)$. Thus if $P$ is a group of permutations on the set of $k$ piles, we define
$$\sh(P,n) := \langle \rho_\tau \sigma  \mid \tau \in P \rangle \leqslant \sym(kn)$$
which we call a \emph{generalised shuffle group}. Note that since the identity permutation is in $P$, we have $\sh(P,n ) = \langle \sigma, \rho_\tau \mid \tau \in P \rangle$.
In this notation, the group described by Theorem~\ref{dgk} is $\sh(\sym(2),n)$. 

For a positive integer $m$, we write $[m]=\{0,\dots,m-1\}$ and use this as a set of residues modulo $m$. It is convenient to denote the set of piles by $[k]$ and the set of cards by $[kn]$. 
The case where $P=\sym(k)$ was considered by Medvedoff and Morrison \cite{MM} (they investigated the group $\sh(\sym(k),n)$, which they called $G_{kn,n}$). As a generalisation of the case where $n=2^f$ in Theorem~\ref{dgk}, they showed:


\begin{thm} \cite[Theorem 2]{MM}
\label{thm:MMprod}
Suppose that $n = k^f$ for some positive integers $k$  and $f$ with $k \geqslant 2$. Then the elements of $[kn]$ can be identified with the set $[k]^{f+1}$ of $(f+1)$-tuples of elements of $[k]$, and $\sh(\sym(k),k^f) = \sym(k) \wr C_{f+1}$ in product action on $[k]^{f+1}$.
\end{thm}

Depending on the congruences of $k$ and $n$ modulo four, it is easy to decide whether $\sh(\sym(k),n)$ is contained in $\alt(kn)$ or not -- see \cite[Theorem 1]{MM}. Computations by the authors of \cite{MM} for $k=3$ and $k=4$, for small values of $n$, showed that when $n$ is not a power of $3$ or $2$ respectively, the shuffle group $\sh(\sym(k),n)$  is actually equal to the symmetric group $\sym(kn)$, or to $\alt(kn)$ if it is contained there. Based on this evidence, they posed conjectures for the structure of the group when $k=3$ or $k=4$. Later in \cite{CHMW}, the case of $k=4$ and $n=2^{f}$ for $f$ an odd integer was examined further (the case $n=2^{2f}=4^f$ being covered by the above theorem). The main result \cite[Theorem 2.6]{CHMW} shows that  $\sh(\sym(4),2^f)$ is the full affine group of degree $2^{f+2}$  when $f$ is an odd integer. The authors of \cite{CHMW} then made the following conjecture for general $k$, which  could be summarised as saying ``\emph{$\sh(\sym(k),n)$ is as large as possible}''.
%

\begin{conj} \cite[Conjecture 1.1]{CHMW}
\label{conj:CHMW}
Assume that $k \geqslant 3$, that $n$ is not a power of $k$ and if $k=4$ that $n \neq 2^f$ for any $f$. Let $k_{(4)}, n_{(4)} \in [4]$ be the residues of $k$ and $n$ modulo $4$ respectively.
  Then
	\[ \sh(\sym(k),n) = \left\{ \begin{aligned} &\alt(kn) &&\text{if $n_{(4)} = 0$ or $(k_{(4)}, n_{(4)}) \in \{ (0,2), (1,2) \}$} \\ &\sym(kn) &&\text{if $n$ is odd or $(k_{(4)}, n_{(4)}) \in \{ (2,2), (3,2) \}$}. \end{aligned} \right. \]
\end{conj}

We were motivated to study the validity of this conjecture.  More generally, we wondered what might be said about the structure of $\sh(P,n)$ for any subgroup $P$ of $\sym(k)$. It is reasonable to guess that Conjecture \ref{conj:CHMW} might be true, as a plethora of results show that it is ``easy'' to generate a subgroup of $\sym(m)$ that contains $\alt(m)$. For example, a result of Dixon \cite{dixon} shows that as $m \rightarrow \infty$, the probability that a randomly chosen pair of elements of $\sym(m)$ generate a subgroup that contains $\alt(m)$ tends to $1$.  On the other hand, the groups that we consider here are far from randomly generated. In the next section we describe our progress.

\subsection{Main results}
 
To frame our first main result, we briefly discuss \emph{structures preserved by permutation groups}. For a prime $p$, the set $[p^e]$ is in bijection with the vector space $ \mathbb F_p^e$. This gives an embedding of $\AGL(e,p) = C_p^e \rtimes \GL(e,p)$, the group of affine transformations of $ \mathbb F_p^e$, into $\sym(p^e)$. We say that a  permutation group $P\leqslant \sym(p^e)$ \emph{preserves an affine structure} on $[p^e]$ if $P$ is   a subgroup of $\AGL(e,p)$ (relative to some some bijection $[p^e]\rightarrow \mathbb{F}_p^e$). For integers $\ell$ and $e$, the set $[\ell^e]$ is in bijection with the cartesian product $[\ell]^e$. On the latter set, the wreath product $\sym(\ell) \wr \sym(e)$ acts in \emph{product action}: if $(g_1,\dots,g_e)\in \sym(\ell)^e$ then $(g_1,\dots,g_e) : (\omega_1,\dots,\omega_e) \mapsto (\omega_1^{g_1},\dots,\omega_e^{g_e})$ and if $\sigma \in \sym(e)$ then $
\sigma : (\omega_1,\dots,\omega_e) \mapsto (\omega_{1^{\sigma^{-1}}},\dots,\omega_{e^{\sigma^{-1}}}).$ As in the affine case, we say that a permutation group $P \leqslant \sym(\ell^e)$ \emph{preserves a product structure} on $[\ell^e]$ if $P$ is   a subgroup $\sym(\ell) \wr \sym(e)$ acting in product action (for some bijection). There is also  the \emph{imprimitive action} of a wreath product $\sym(\ell) \wr \sym(e)$ on the set $[\ell] \times [e]$ where $(g_1,\dots,g_e)\sigma : (a,b) \mapsto (a^{g_b},b^\sigma)$. This action is imprimitive whenever both $\ell > 1$ and $e>1$.
 
 Theorem~\ref{thm:MMprod} implies that $\sh(\sym(k),k^f)$ preserves a product structure on $[k^{f+1}]$, and \cite[Theorem 2.6]{CHMW} shows that $\sh(\sym(4),2^f)$ preserves an affine structure on $[2^{f+2}]$.   Since $\sym(4)= \mathrm{AGL}(2,2)$ is an affine group, and we can regard $\sym(k)$ as preserving the ``trivial'' product structure  $[k]^1$,  the aforementioned results may be interpreted as saying that, if $P$ preserves a structure on $[k]$, then for appropriate $n$ the shuffle group $\sh(P,n)$ also preserves such a structure on $[kn]$.  In our first result, Theorem~\ref{thm:productid}, we show that this is the case in general.

\begin{thm} \label{thm:productid}
Let $\ell$, $e$ and $f$ be positive integers with  $\ell \geqslant 2$. The following hold:
	\begin{enumerate}
	\item If $e \mid f$, then $\sh(P,\ell^f) = P \wr C_{1+f/e}$, for any $P \leqslant \sym(\ell^e)$.
	\item \label{pa} If $e\nmid f$ and  $P = \sym(\ell) \wr \sym(e)$ acts in product action on $[\ell]^{e}$, then $\sh(P, \ell^f)$ acts in product action on $[\ell^{e+f}]$ and $ \sh(P, \ell^f) =\sym(\ell) \wr \sym(e+f)$.
	\item If $e\nmid f$, $\ell$ is prime and  $P=\AGL(e,\ell)$, then $\sh( P, \ell^f) =  \AGL(e+f,\ell)$.
	\end{enumerate}
\end{thm}

Theorem~\ref{thm:productid} is proved in Section~\ref{sec:productid} where we also make precise statements about  the  structure of $\sh(P, \ell^f)$ for $P\leqslant \sym(\ell) \wr \sym(e)$ and for $P \leqslant \AGL(e,\ell)$. We have the following application of Theorem~\ref{thm:productid} which establishes Conjecture~\ref{conj:CHMW} for a doubly infinite family of integers.

\begin{cor}
\label{cor:productid}
Suppose that $k=\ell^e$ and $n=\ell^f$ for some positive integers $\ell$, $e$ and $f$ with $\ell \geqslant 2$. If $f$ is not a multiple of $e$ and $k\neq 4$, then $\sh(\sym(k),n) = \alt(kn)$ if $\ell$ is even and $\sh(\sym(k),n) = \sym(kn)$ if $\ell$ is odd.
\end{cor}



We note  that Corollary~\ref{cor:productid} with $(k,n) = (9,3)$ is proved in \cite[Theorem 3.6]{CHMW}, using different methods.
In our goal of proving Conjecture \ref{conj:CHMW}, we are motivated to consider properties of  $\sh(P,n)$ for $P \leqslant \sym(k)$ that \emph{should} hold if the conjecture is valid. The most basic concerns transitivity, and is established in Lemma~\ref{lemma:transitive1}. A more important notion is that of primitivity, which we address below.

\begin{thm} \label{thm:prim}
If $P$ is primitive and non-regular on $[k]$, then $\sh(P,n)$ is primitive on $ [kn]$.
Moreover, $\sh(\sym(k),n)$ is primitive on $[kn]$ if and only if $k\geqslant3$.
\end{thm}

The non-regularity assumption for $P$ cannot be removed in general, since  $P=\sym(2)$ is primitive and regular on two points and Theorem~\ref{dgk} implies that, for all $n$, the group $\sh(\sym(2),n)$ is imprimitive. Similarly the primitivity assumption on $P$ cannot be relaxed, since if $P$ is an imprimitive group of degree $k$, then Theorem~\ref{thm:productid} shows that  $\sh(P,k^f) = P \wr C_{f+1}$   in product action on $[k]^{f+1}$  -- and this action is imprimitive whether $P$ is non-regular or not. Hence both parts of the hypothesis are necessary.  We note also that, if $P=C_k$ for an odd prime $k$, so that $P$ is primitive and regular on $[k]$, then if $n=k^f$, Theorem~\ref{thm:productid} implies that $\sh(C_k,n)\leqslant \sh(\AGL(1,k),n)=\AGL(1,k)\wr C_{1+f}$. Now $\AGL(1,k)\wr C_{1+f}$ is imprimitive on $[kn]$ since the diagonal subgroup of $C_k^{1+f}$ is an intransitive normal subgroup, and therefore also $\sh(C_k,n)$ is imprimitive in this case. However we have verified computationally that, if $k\leqslant 13$, $k<n\leqslant 1000$,
and $n$ is not a power of $k$, then $\sh(C_k,n)$ contains $\alt(kn)$ (see Section~\ref{sec: summary}). 

\begin{conj}\label{conj1}
If $k$ is an odd prime, $n>k$, and $n$ is not a power of $k$, then $\sh(C_k,n)$ contains $\alt(kn)$.
\end{conj}

If Conjecture~\ref{conj1} is true then for $n\neq k^f$ we would have a stronger form of Theorem~\ref{thm:prim}, but even our current 
Theorem~\ref{thm:prim} has a number of applications in this work, notably in Section~\ref{sec:cascading}.

The next property we sought to explore is that of 2-transitivity. In part (1) of the theorem below we see that Conjecture~\ref{conj:CHMW} holds whenever $k>n$, that is, when there are more piles than cards in a pile.   A group $G$ is almost simple if there is a nonabelian finite  simple group  $T$ such that $G$ is isomorphic to a subgroup of $\aut(T)$ containing the inner automorphism group of $T$. A classical result of Burnside says that 2-transitive groups are either affine or almost simple, \cite[Theorem 4.1B]{dixmort}.  The latter have been classified using the Classification of the Finite Simple Groups; we have reproduced a list of them in Section~\ref{sec:2trans}.

\begin{thm} \label{thm:2trans}
Suppose that $k>n \geqslant 2$ and that $P\leqslant \sym(k)$ is $2$-transitive. Let $k_{(4)}, n_{(4)} \in [4]$ be such that $k \equiv k_{(4)} \pmod{4}$ and $n \equiv n_{(4)} \pmod{4}$. Then $\sh(P,n)$ is $2$-transitive and the following hold:
	\begin{enumerate}
	\item \label{2tr-alt} If $P$ is $\alt(k)$ or $\sym(k)$ then either $(k,n) = (4,2)$ and $\sh(P,n) = \AGL(3,2)$,  or $(k,n) \neq (4,2)$ and $\sh(P,n) = \alt(kn)$ or $\sym(kn)$. (Details in Lemma $\ref{lem:2trans-full}$.) 
	\item \label{2tr-P-AS} If $P$ is almost simple, then also $\sh(P,n)$ is almost simple.
	\item If $P$ is affine with $k = p^e\geqslant3$ for some prime $p$ and positive integer $e$, then either $n = p^f$ for some positive integer $f$ and $\sh(P,n)$ is affine, or $n \neq p^f$ for any $f$ and $\sh(P,n)=\alt(kn)$ or $\sym(kn)$.
	\end{enumerate}
\end{thm}

The condition that $k$ is greater than $n$ is necessary in Theorem~\ref{thm:2trans}, for Theorem ~\ref{thm:productid} shows that  $\sh(\sym(k),k) = \sym(k)\wr C_2$, which is not 2-transitive on $[k^2]$.

\subsection{Cascading shuffle groups}

Different values of $k$ and $n$ with the same product $kn$ can yield different shuffle groups acting on the same set. A priori these groups have nothing to do with each other, yet we have found that for certain values of $k$, there is a deep connection between the groups.

Fix $k=2^e$ for some integer $e \geqslant 2$ and let $n$ be arbitrary. Then, for each $t\in \{1,2,\dots,e\}$, define $V_t$ to be the elementary abelian group of order $2^t$ acting regularly on the set $[2^t]$. Further, set $G_t=\sh(V_t, 2^{e-t}n)$, so that each of the $e$ shuffle groups $G_1$, \dots, $G_e$ acts on $2^en$ points. Note that $G_1= \sh(\sym(2),2^{e-1}n)$ is a shuffle group appearing in the theorem of Diaconis  et ~al, Theorem~\ref{dgk}. 

\begin{thm}
\label{thm:k=2^e}
Suppose that $n$ is not a power of $2$ and $k=2^e$ with $e\geqslant 2$. The  following hold:
\begin{enumerate}
\item if $(k,n)=(4,3)$, then $G_1 = C_2^6\rtimes \sym(5)$ and $G_2 = C_2^5 \rtimes \sym(5)$;
\item if $(k,n)=(4,6)$, then $G_1 =G_2$ and $G_1 \cong C_2^{11} \rtimes M_{12}$;
\item if $(k,n)=(8,3)$, then $G_1 =G_2= G_3$ and $G_1 \cong C_2^{11} \rtimes M_{12}$;
\item if $k=4$ and $n\geqslant 5$ is odd, then $G_1 = C_2 \wr \sym(2n)$ and $G_2 = \ker(\mathrm{sgn})$;
\item in all other cases, $G_1 = G_2 = \dots = G_e$ and $G_1 =\ker(\mathrm{sgn}) \cap \ker( \overline{\mathrm{sgn}})$.
\end{enumerate}
\end{thm}

Although the connection between these cascading shuffle groups is striking in its own right, our principal goal for investigating their structure is to provide tools to analyse their overgroups. In particular, since $G_e = \sh(V_e,n) \leqslant \sh(\sym(k),n)$, we are able to use the previous theorem to obtain:

\begin{cor}
\label{cor:k=2^e}
 Suppose that $n$ is not a power of $2$ and that $k = 2^e$ for some $e\geqslant 2$. Then $\sh(\sym(k),n)$ is the  alternating group  or the symmetric group if $n$ is even or odd, respectively.
\end{cor}

The results on cascading shuffle groups are proved in Section~\ref{sec:cascading}. Using Theorem~\ref{dgk} for $k=2$, Theorem~\ref{thm:MMprod} and \cite[Theorem 2.6]{CHMW} for $k=4$ and $n$ a power of 2,  Corollary~\ref{cor:productid} for $k=2^e \geqslant 8$ and $n$ a power of 2 and Corollary~\ref{cor:k=2^e} for $k=2^e \geqslant 4$ and $n$ not a power of 2, the structure of the shuffle group $\sh(\sym(2^e),n)$ is now known for all integers $e$ and $n$.

\subsection{Summary and questions}\label{sec: summary}

Concerning Conjecture~\ref{conj:CHMW} we have shown that  the conjecture holds, that is, that $\sh(\sym(k),n) \geqslant \alt(kn)$, in the following instances:
\begin{enumerate}
\item $k>n$,
\item $k=2^e$ for some $e\geqslant 2$ and $n \neq 2^f$ for any $f$,
\item $k=\ell^e$ and $n=\ell^f$ for some positive integer $\ell$ where $e \nmid f$ and $k\neq 4$.
\end{enumerate}
The first extension of this work would be to the case that $k=3^e$ and $n \neq 3^f$ for any $f$. Some arguments that we have given in Section~\ref{sec:cascading} might generalise, but we note that for such $n$ there is no information concerning the structure of $\sh(\sym(3),3^{e-1}n)$ and thus the endgame to exploit the potential connection between $\sh(\sym(3^e),n)$ and    $\sh(\sym(3), 3^{e-1}n)$ cannot yet be achieved.

Concerning computations,  at several points we have made use of  {\sc Magma}  \cite{magma}. We have verified 
the following remarkable fact: for $k \in \{3,\dots,13\}$ and $n \in \{k+1,\dots,1000\}$ with $n \neq k^f$ for any $f$, we have that $\sh(C_k,n)$ contains $\alt(kn)$.
This suggests that Conjecture~\ref{conj:CHMW} might be ``overwhelmingly true''.
%

Considering what might be the converse of the conjecture, if $P \leqslant \sym(k)$ is such that $\sh(P,n)$ contains $\alt(kn)$, then $P$ need not even be transitive. There are \emph{many} examples,  for small values of $k$ and $n$, of intransitive groups $P$ with  $\alt(kn)\leqslant \sh(P,n)$. On the other hand, in Remark~\ref{rem:intrans} we give infinitely many easy examples of nontrivial intransitive subgroups $P$ with $\sh(P,n)$  intransitive on $[kn]$.
Thus the precise conditions on $k$, $n$ and  $P \leqslant \sym(k)$ to guarantee that $\sh(P,n)$ is not the full alternating group or symmetric group seem to be quite delicate.

\begin{qn} 
Is it possible to classify the pairs $(P,M)$ such that $P<\sym(k)$ and $M<\sym(kn)$ are both maximal subgroups, and $\sh(P, n)\leq M$?
\end{qn}

Theorems~\ref{thm:productid} and~\ref{thm:2trans} give some examples of $(P,M)$. An answer in the general case would be very interesting.

\section*{Acklowledgements}
We thank Kent Morrison for sharing the slides of his talk at the 2009 MathFest on this subject and general enthusiasm about our work.

\section{Preliminaries}

Throughout this paper we use the following notation: for any positive integer $u$, let
$$ [u] = \{0, 1, \ldots, u-1\}$$
and let $\sym(u)$ and $\alt(u)$ denote the symmetric and alternating groups on the set $[u]$.
Let $k$ and $n$ be fixed integers greater than $1$, and let $[kn]$ be a deck of  cards divided into $k$ piles with $n$ cards each. Thus the cards are labelled  $0, \dots, kn-1$ and the piles as $0, \ldots, k-1$ (from left to right, say), with $an, an+1, \ldots, (a+1)n-1$ (from top to bottom) being the cards in pile $a$. Thus each card in $[kn]$ has a unique label $an+b$ for some $a \in [k]$ and $b \in [n]$.

Suppose that the piles are arranged in a single row. The \emph{standard shuffle} is performed by picking up the top card from each of the piles from left to right and repeating until all cards have been picked up. As a permutation on $[kn]$ we denote the standard shuffle by $\sigma$. Let $i \in [kn]$ and suppose that $i = an + b$, with $a \in [k]$ and $b \in [n]$. Then the card $i$ is the $b$th card in the $a$th pile; thus under the shuffle $\sigma$, the $bk$ cards in the top $b$ rows have already been picked up prior to card $i$ and $i$ is the $a$th card in its row to be picked up. Hence, taking $[n]$ and $[k]$ as sets of residues modulo $n$ and $k$, respectively, we have
	\begin{equation} \label{eq: s}
	\sigma \ : \ i \mapsto bk + a = (i \pmod{n}) \cdot k + \left\lfloor \frac{i}{n} \right\rfloor.
	\end{equation}

Each permutation $\tau$ of the $k$ piles induces a permutation $\rho_\tau$ of the $kn$ cards. Since $\rho_\tau$ does not change the order of the cards within each pile, it follows that for each $a \in [k]$ and $b \in [n]$,
	\begin{equation} \label{eq: rho}
	(an+b)^{\rho_\tau} = a^\tau n + b.
	\end{equation}
If $P$ is a permutation group on the set of piles, we denote by $P^\rho$ the group $\{ \rho_\tau \ | \ \tau \in P \}$ of corresponding permutations of $[kn]$ (in effect, $\rho$ may be thought of as an embedding of $P$ into $\sym(kn)$).
This leads to the following definition.

\begin{defn}\label{def:sh}
Let $k$ and $n$ be integers greater than $1$, view $[kn]$ a deck of $kn$ cards divided into $k$ piles with $n$ cards each, and let $P \leqslant \sym(k)$. The \emph{generalised shuffle group} on $[kn]$ is the group $\sh(P,n)$ defined by
	\[ \sh(P,n):= \langle \sigma, \rho_\tau \mid \tau \in P \rangle. \]
\end{defn}


The following is fundamental.

\begin{lem} \cite[Proposition 1]{MM} \label{lem:shuffle acts}
The shuffle $\sigma$ fixes both $0$ and $kn-1$. For $i \in [kn]$ with $i \neq kn-1$ we have
 \[ \sigma : i \mapsto ki \pmod{kn-1} \]
taking $\{0, \dots, kn-2\}$ to be a set of residues modulo $kn-1$.
\end{lem}

\begin{proof}
Since the first card is labelled $0$ and is picked up first, we have $0^\sigma = 0 \equiv k \cdot 0 \pmod{kn-1}$. Similarly, the last card is labelled $kn-1$ and is picked up last, so $(kn-1)^\sigma = kn-1$. Now let $i \in [kn]$ with $i \neq 0, kn-1$. Let $a,b \in \mathbb{N}$ be such that $i = an+b$. Then $i^{\sigma} = a+bk$ by (\ref{eq: s}). Now observe that, since  $i \neq kn-1$, at least one of $a < k-1$ or $b < n-1$ holds and hence $a+bk < (k-1) + (n-1)k = kn-1$, so $a+bk = i^\sigma \pmod{kn-1}$. Using this fact, we observe that $ik = (an+b)k = akn + bk \equiv a + bk \pmod{kn-1}$. Thus we have established that $i^{\sigma} \equiv ki \pmod{kn-1}$ as required.
\end{proof}

\begin{lem}[{\cite[Lemmas 1 and 2]{MM}}]
\label{lem:parity} The following hold:
	\begin{enumerate}
	\item The shuffle $\sigma \in \alt(kn)$ if and only if either $n_{(4)} \in \{0,1\}$ or $k_{(4)} \in \{0,1\}$.
	\item For any $\tau \in \sym(k)$, $\rho_\tau \in \alt(kn)$ if and only if either $n$ is even or $\tau \in \alt(k)$.
	\end{enumerate}
\end{lem}

\begin{cor} \label{cor:parity}
$\sh(P,n) \leqslant \alt(kn)$ if and only if one of the following holds:
	\begin{enumerate}
	\item $n_{(4)} = 0$, or
	\item $n_{(4)} = 2$ and $k_{(4)} \in \{0,1\}$, or
	\item $n$ is odd, $P \leqslant \alt(k)$, and either $n_{(4)} = 1$ or $k_{(4)} \in \{0,1\}$;
	\end{enumerate}
\end{cor}

We extend \cite[Proposition 3]{MM} to the case where $P$ is any transitive subgroup of $\sym(k)$ (so the proof cannot use arbitrary permutations in $\sym(k)$).

\begin{lem} \label{lemma:transitive1}
If $P$ is a transitive subgroup of $\sym(k)$, then $\sh(P,n)$ is a transitive subgroup of $\sym(kn)$.
\end{lem}

\begin{proof}
Let $G = \sh(P,n)$ and let $i = an+b$ for $a \in [k]$ and $b \in [n]$.  We claim that there is an element $g \in G$ such that $0^g = b$. If the claim is true, then since $P$ is transitive, there is $\tau \in P$ such that $0^\tau = a$ (recall that $\tau$ acts on $[k]$). Hence $0^{g \rho_\tau} = b^{\rho_\tau} = (0 \cdot n + b)^{\rho_\tau}  = (0^\tau)n + b = an+b = i$ (using \eqref{eq: rho}). Thus if the claim holds then $G$ is transitive.

We now prove the claim. Let $m$ be such that $k^m \leqslant n-1 < k^{m+1}$. Since $b \in [n]$ we may write $b = \sum_{r=0}^m i_r\, k^r$, where $0 \leqslant i_r \leqslant k-1$ for each $r$. Since $P$ is transitive on $[k]$, for each $r \in \{0, \ldots, m\}$ there is an element $\tau_r \in P$ such that $0^{\tau_r} = i_r$. Now for any  $c \in [n]$, it follows from \eqref{eq: rho} that 
$c^{\rho_{\tau_r}} = (0 \cdot n + c)^{\rho_{\tau_r}} = i_r n + c$, so that, using \eqref{eq: s}, 
	\begin{equation} \label{eq: taursigma}
	(c^{\rho_{\tau_r}})^ \sigma  = (i_r n + c)^\sigma = ck + i_r.
	\end{equation}
If $m = 0$ then $b = i_0$ and $0^{\rho_{\tau_0}\sigma} = i_0 = b$, as required. If $m \geqslant 1$ then for any $\ell \leqslant m-1$,
	\[ \sum_{r=0}^\ell i_{r+m-\ell}\, k^r \leqslant k^m - 1 < n-1 \]
and for $\ell = m$,
	\[ \sum_{r=0}^\ell i_{r+m-\ell}\, k^r = \sum_{r=0}^\ell i_r\, k^r = b \leqslant n-1. \]
Using this fact together with (\ref{eq: rho}) and (\ref{eq: taursigma}) we have
	\begin{align*}
	0^{(\rho_{\tau_m}\sigma) \dots (\rho_{\tau_0}\sigma)}
	&= (i_m)^{(\rho_{\tau_{m-1}}\sigma) \dots (\rho_{\tau_0}\sigma)} \\
	&= (i_m k + i_{m-1})^{(\rho_{\tau_{m-2}}\sigma) \dots (\rho_{\tau_0}\sigma)} \\
	&\vdots \\
	&= \left(\sum_{r=0}^\ell i_{r+m-\ell}\, k^r\right)^{(\rho_{\tau_{m-\ell-1}}\sigma) \dots (\rho_{\tau_0}\sigma)} \\
	&\vdots \\
	&= \left (\sum_{r=0}^{m-1} i_{r+1}\, k^r \right )^{ (\rho_{\tau_0}\sigma)} \\
	&= \sum_{r=0}^m i_r\, k^r \\
	&= b
	\end{align*}
This establishes the claim, and hence the lemma.
\end{proof}
%
%
%

\begin{rem}
\label{rem:intrans}
A shuffle group can be intransitive. Let $k$ and $n$ be arbitrary. If $P\leqslant \sym(k)$ is such that $0^\tau = 0$ for all $\tau \in P$, then in the action of $P^\rho$ on $[kn]$ we have $0^{\rho_\tau} = 0$ for all $\tau \in P$. Since $0^\sigma = 0$, the group $\sh(P,n)$ fixes $0$, and is therefore intransitive.
\end{rem}

\subsection{Permutation group concepts and results}\label{sub:perm}

Let $A$ be a permutation group acting on a finite set $\Delta$. For an element $a \in A$, we denote by $\supp_\Delta(a)$ the set of points of $\Delta$ that are moved by $a$, that is, points $\delta \in \Delta$ such that $\delta^a \ne \delta$, and we call $|\supp_\Delta(a)|$ the \emph{degree} or \emph{support size} of $a$. The \emph{minimal degree} of $A$, denoted $\textrm{mindeg}_\Delta(A)$, is the minimum of $|\supp_\Delta(a)|$ over all non-identity elements $a \in A$. The action of $A$ is said to be \emph{semiregular} if the stabiliser of any point is the identity.

Let $\Sigma = \Delta^c$ for some integer $c \geqslant 2$, and denote the $i$th factor of $\Sigma$ by $\Delta_i$. Each of the groups $A^c$ and $\sym(c)$ has a natural action on $\Sigma$: the former by acting componentwise, and the latter by permutation of the coordinates $i$. The \emph{product action} of the wreath product $A \wr \sym(c) = A^c \rtimes \sym(c)$ on $\Sigma$ is the combination of these actions of $A^c$ and $\sym(c)$. More precisely, for any ${\boldsymbol\delta} = (\delta_1, \ldots, \delta_c) \in \Sigma$, $g = (a_1, \ldots, a_c) \in A^c$, and $\rho \in \sym(c)$,
	\[ 
	{\boldsymbol\delta}^{g\rho} = \big(\delta^{a_{1'}}_{1'}, \ldots, \delta^{a_{c'}}_{c'}\big) 
	\]
where $i' = i^{\rho^{-1}}$ for each $i$.

\begin{lem} \label{lem:mindeg}
Let $\Delta$ be a set with $|\Delta| = d$, and $A \leqslant \sym(\Delta)$ such that $A\ne 1$. Let $X = A \wr \sym(c) = A^c \rtimes \sym(c)$ with $c \geqslant 2$, in product action on $\Sigma = \Delta^c$. Let $m = \textrm{mindeg}_\Delta(A)$. Then
\[
\textrm{mindeg}_\Sigma(X) =  \left \{ \begin{array}{l l }
	 d^{c-1}m & \text{ if $A$ is   not semiregular,} \\ 
	 d^{c-1}(d-1) &  \text{ if $A$ is   semiregular. } \end{array} \right .
	 \]
In particular $\textrm{mindeg}_\Sigma(X) \leqslant d^{c-1}(d-1)$.
\end{lem}

\begin{proof}
Note that $m \leqslant d-1$ if $A$ is not semiregular on $\Delta$. First we exhibit an element of $X$ with support size $d^{c-1}m$ when $A$ is not semiregular. By definition $A$ contains an element with support size $m$ in $\Delta$. Let $a$ be such an element, and let $g := (a,1,\dots,1) \in A^c \leqslant X$. Then a point ${\boldsymbol\delta} = (\delta_1, \dots, \delta_c) \in \Sigma$ is moved by $g$ if and only if $\delta_1^a \ne \delta_1$. Thus there are no restrictions on the $\delta_i$ for $i > 1$ but we must have $\delta_1 \in \supp_\Delta(a)$, and hence $|\supp_\Sigma(g)| = d^{c-1}m$.

Next we show that each $g = (g_1, \dots, g_c) \in A^c \setminus \{1\}$ has support size at least $d^{c-1}m$ in $\Sigma$. This holds if $A$ is semiregular on $\Delta$ since in this case $m=d$ and $A^c$ is semiregular on $\Delta^c$. If $A$ is not semiregular, then since $g \ne 1$, we must have $g_i \ne 1$ for some $i$, and without loss of generality we may assume that $i=1$. Thus $g_1 \ne 1$ so $m':= |\supp_\Delta(g_1)| \geqslant m$. As in the previous paragraph, each point  ${\boldsymbol\delta} = (\delta_1, \dots, \delta_c) \in \Sigma$ such that $\delta_1 \in \supp_\Delta(g_1)$  is moved by $g$, and this implies that $|\supp_\Sigma(g)| \geqslant d^{c-1}m' \geqslant d^{c-1}m$.

Finally we examine the supports of elements $g \in X \setminus A^c$. Such elements have the form $g = (g_1, \dots, g_c)\alpha$ with the $g_i \in A$ and $\alpha$ a non-identity element of $\sym(c)$. Since $\alpha \ne 1$, there exist distinct $i, j \in \{1, \dots, c\}$ such that $j=i^\alpha \ne i$. For ${\boldsymbol\delta} = (\delta_1, \dots, \delta_c) \in \Sigma$, we have
	\[ 
	{\boldsymbol\delta}^g = \big(\delta_{1\alpha^{-1}}^{g_{1\alpha^{-1}}}, \dots, \delta_{c\alpha^{-1}}^{g_{c\alpha^{-1}}}\big). 
	\]
The $j^{th}$ entry of ${\boldsymbol\delta}^g$ is $\delta_{j\alpha^{-1}}^{g_{j\alpha^{-1}}} = \delta_i^{g_i}$, so if $\delta_j \ne \delta_{i}^{g_i}$ then ${\boldsymbol\delta}^g \ne {\boldsymbol\delta}$ (regardless of the choices of the $\delta_u$ for $u \ne j$). Hence $g$ has support size at least $d^{c-1}(d-1)$. If $A$ is not semiregular then this is at least $d^{c-1}m$ because $m \leqslant d-1$, and hence $\textrm{mindeg}_\Sigma(X) = d^{c-1}m$.
If $A$ is semiregular then   all $g\not\in A^c$ have support size at least $d^{c-1}(d-1)$, and if $g=\alpha$ is the transposition $(i,j)$ then the discussion shows that $g$ has support size precisely $d^{c-1}(d-1)$, and hence $\textrm{mindeg}_\Sigma(X) = d^{c-1}(d-1)$. 
\end{proof}

\subsection{Primitive prime divisors} We need the following results in Section~\ref{sec:2trans}.
For a prime $p$ and positive integer $n$, we denote by $n_p$ the highest power of $p$ dividing $n$; $n_p$ is called the \emph{$p$-part} of $n$. For a prime power $q$ and an integer $d$, the prime $p$ is called a primitive prime divisor  of $q^d-1$ if $p$ divides $q^d-1$ and does not divide $q^e-1$ for any $e<d$.  

\begin{lem}[Zsigmondy's Theorem \cite{zsig}]
\label{lem:zsig}
Let $q$ be a prime power. Then there exists a primitive prime divisor of $q^d-1$, unless $(q,d)=(2,6)$ or $q$ is odd, $d=2$ and $q+1 = 2^c$ for some $c$.
\end{lem}

\begin{lem}[{\cite[Lemma 4.1]{involutions}, \cite[Lemma 4.1]{abundant}}]\label{lemma:p-part}
Let $q$ be a prime power, $p$ a prime not dividing $q$, and $e$ the smallest positive integer such that $p$ divides $q^e-1$. Then $p$ divides $q^d-1$ if and only if $e$ divides $d$. Furthermore:
	\begin{enumerate}
	\item \label{p=2} If $p=2$, then $(q^d-1)_2 = (q^{d_2}-1)_2 = \left\{ \begin{aligned} &(q^2-1)_2 \cdot (d/2)_2 &&\text{if $d$ is even} \\ &(q-1)_2 &&\text{if $d$ is odd.} \end{aligned} \right.$
	
	\item \label{p>2} If $p \neq 2$ and $e$ divides $d$, then $(q^d-1)_p = (q^e-1)_p \cdot (d/e)_p$.
	\end{enumerate}
\end{lem}

\begin{rem} \label{rem:psylcyclic}
For a primitive prime divisor $p$ of $q^d-1$, the Sylow $p$-subgroups of $\GL(d,q)$ are cyclic. Indeed, the Sylow $p$-subgroups are contained  in the so-called Singer cycle subgroups which are cyclic of order $q^d-1$ (see \cite[Kapitel I, 7.3]{HuppertI}). This observation also applies to the simple sections of $\GL(d,q)$, such as $\PSL(d,q)$, $\mathrm{PSp}(d,q)$ (if $d$ is even) and $\PSU(d,q^{1/2})$ (if $q$ is a square), etc.
\end{rem}

\section{Product identification in the ``power case''} \label{sec:productid}

In this section we prove Theorem \ref{thm:productid}.
Let us say that we are in the \emph{power case} if there is an integer $\ell \geqslant 2$ such that $k = \ell^e$ and $n = \ell^f$ for some positive integers $e$ and $f$.

In the power case we can define an identification of $[kn]=[\ell^{e+f}]$ with $[\ell]^{e+f}$ as follows. For an integer $m$ and  $i \in [\ell^m]$, there are unique integers $ i_r \in [\ell]$ such that 
	\begin{equation} \label{eq:base-l}
	i = \sum_{r=0}^{m-1} i_r\,\ell^{m-1-r }
	\end{equation}
 and we   identify $i\in [\ell^m]$ with an element in $[\ell]^m$ as follows:
	\begin{equation} \label{eq:product-id}
	i \sim (i_0,\dots,i_{m-1}).
	\end{equation}

\begin{lem} \label{lem:pa-shuffle}
Assume that we are in the power case with $k = \ell^e$ and $n = \ell^f$.
The shuffle respects the identification of $[\ell^{e+f}]$ with $[\ell]^{e+f}$ in \eqref{eq:product-id} and for each $i\in [kn]$ we have
	\[ i^\sigma \sim \big(i_e, \ldots, i_{e+f-1}, i_0, \ldots, i_{e-1}\big). \]
That is, the effect of $\sigma$ is to shift the coordinates of any card $e$ places to the left (modulo $e+f$). Moreover, the order of $\sigma$ is $\frac{e+f}{\mathrm{gcd}(e,f)}$.
\end{lem}

\begin{proof}
Let $i \in [kn]$ and let $i \sim (i_0,\dots,i_{e+f-1})$ as in \eqref{eq:product-id}. We compute the effect of the shuffle $\sigma$ on $i$. If $i = 0$ or $i = kn-1$, then $i \sim (0,\dots,0)$ or $i \sim (\ell-1,\dots,\ell-1)$, respectively. Now $i^\sigma = i$, and since all of the coordinates of the representation of $i$ as in \eqref{eq:product-id} are equal, we see that shifting the entries to the left by $e$ places (modulo $e+f$) also fixes this representation of $i$. Thus we may assume that $0 \neq i \neq kn-1$. By Lemma~\ref{lem:shuffle acts} we have $i^\sigma \equiv ik \pmod{\ell^{e+f} - 1}$. Now $i=\sum_{r=0}^{e+f-1} i_r \ell^{e+f-1-r}$, and to evaluate $ik$, we write $i = A+B$ where $A = \sum_{r=0}^{e-1} i_r \ell^{e+f-1-r}$ and $B=\sum_{r=e}^{e+f-1} i_r \ell^{e+f-1-r}$. First, 
$$kA = \ell^eA = \ell^{e+f} \sum_{r=0}^{e-1}i_r\, \ell^{e-1-r} \equiv \sum_{r=0}^{e-1} i_r\, \ell^{e-1-r} \pmod{\ell^{e+f}-1}.$$
Now $\sum_{r=0}^{e-1} i_r\, \ell^{e-1-r} = \sum_{r=f}^{e+f-1} i_{r-f \pmod{e+f}}\, \ell^{e+f-1-r} = \sum_{r=f}^{e+f-1}i_{e+r \pmod{e+f}}\ell^{e+f-1-r}$.
Secondly, we write $B = \sum_{r=e}^{e+f-1} i_r \, \ell^{e+f-1-r} = \sum_{r=0}^{f-1} i_{e+r}\, \ell^{f-1-r} $ and then
$$
kB = \ell^e B = \sum_{r=0}^{f-1} i_{e+r} \, \ell^{e+f-1-r}.
$$
Hence $ik \equiv \sum_{r=0}^{e+f-1} i_{e+r \pmod{e+f}}\, \ell^{e+f-1-r} \pmod{\ell^{e+f}-1}$, and so
 $$
 i^\sigma \sim (i_e,i_{e+1},\dots,i_{e+f-1},i_0,\dots,i_{e-1}),
 $$
%
as required. For each $j$, $\sigma^j$ permutes the coordinates by moving all coordinates of a tuple to the left by $ej$ places, modulo $e+f$, and so the order of $\sigma$ is the least positive integer $q$ such that $eq$ is a multiple of $e+f$, namely $q = \frac{e+f}{\mathrm{gcd}(e,f)}$.
\end{proof}

\begin{lem} \label{lem:pa-P}
Assume we are in the power case with $k=\ell^e$ and $n=\ell^f$. If $P \leqslant \sym(\ell) \wr \sym(e)$ respects the identification of $[k]$ with $[\ell]^e$ as in \eqref{eq:product-id}, then $P^\rho$ respects the identification of $[kn]$ with $[\ell]^{e+f}$ as in \eqref{eq:product-id} and hence $P^\rho \leqslant \sym(\ell) \wr \sym(e+f)$ in product action on $[\ell]^{e+f}$. Further, if $\tau = (x_0,\dots,x_{e-1}) \alpha \in P$, then $\rho_\tau = (x_0,\dots,x_{e-1},1,\dots,1) \tilde{\alpha}$, where $\tilde{\alpha} \in \sym(e+f)$ is defined by $\tilde{\alpha}: r \mapsto r^\alpha$ for $0 \leqslant r \leqslant e-1$ and $\tilde{\alpha} : r \mapsto r$ for $e \leqslant r \leqslant e+f-1$.
\end{lem}

\begin{proof}
Write each card $i\in [kn]$ in its base-$\ell$ expansion as in (\ref{eq:base-l}), and note that
	\begin{align*}
	i
	&= \sum_{r=0}^{e+f-1} i_r\, \ell^{e+f-r-1} \\
	&= \sum_{r=0}^{e-1} i_r\, \ell^{e+f-r-1} + \sum_{r=e}^{e+f-1} i_r\, \ell^{e+f-r-1} \\
	&= \ell^f \sum_{r=0}^{e-1} i_r\, \ell^{e-r-1} + \sum_{r=e}^{e+f-1} i_r\, \ell^{e+f-r-1} \\
	&= an + b
	\end{align*}
where $a = \sum_{r=0}^{e-1} i_r\, \ell^{e-r-1} \in [k]$ and $b = \sum_{r=e}^{e+f-1} i_r\, \ell^{e+f-r-1} \in [n]$
(as in the proof of Lemma~\ref{lem:pa-shuffle}). By (\ref{eq:product-id}), $a \in[k]$ is identified with the tuple $(i_0, \ldots, i_{e-1})$ in $[\ell]^e$. Let $\tau \in P$. Then $\tau = (x_0, \ldots, x_{e-1}) \alpha$ for some $\alpha \in \sym(e)$ and $x_0, \ldots, x_{e-1} \in \sym(\ell)$. By (\ref{eq: rho}) we have $i ^{\rho_{\tau}} = (an+b)^{\rho_{\tau}} =  (a^\tau)n + b$, where
	\[ a^\tau \sim (i_0, \ldots, i_{e-1})^\tau = \big( (i_{0{\alpha^{-1}}})^{x_{0{\alpha^{-1}}}}, \ldots, (i_{(e-1){\alpha^{-1}}})^{x_{(e-1){\alpha^{-1}}}} \big), \]
and so, with $\tilde{\alpha}$ as in the statement, setting $r' = r^{\tilde{\alpha}^{-1}}$ for each $r \in [e+f]$,
	\begin{align*}
	i^{\rho_\tau}
	&=  a^\tau n + b \\
	&= \left( \sum_{r=0}^{e-1} (i_{r'})^{x_{r'}} \ell^{e-r-1} \right) n + b \\
	&= \sum_{r=0}^{e-1} (i_{r'})^{x_{r'}} \ell^{e+f-r-1} + \sum_{r=e}^{e+f-1} i_r\, \ell^{e+f-r-1}.
	\end{align*}
Hence 
	\[ i^{\rho_\tau} \sim ( (i_{0'})^{x_{0'}},\dots, (i_{(e-1)'})^{x_{(e-1)'}}, i_e, \dots, i_{e+f-1}) \]
and so $\rho_\tau$ is the element $(x_0, \ldots, x_{e-1}, 1, \ldots, 1) \tilde{\alpha}$ of $\sym(\ell) \wr \sym(e+f)$, as required.
\end{proof}

\begin{prop} \label{prop:pa}
Assume we are in the power case with $k=\ell^e$ and $n=\ell^f$. For $Q \leqslant \sym(\ell)$ and $T \leqslant \sym(e)$ let $P = Q \wr T=S\rtimes T$ act in product action on $[\ell^e]$ with respect to the identification in \eqref{eq:product-id} and set $G = \sh(P,n)$.
Let $W = Q \wr \sym(e+f) = X \rtimes Y$ act in product action on $[\ell^{e+f}]$ with respect to the identification in \eqref{eq:product-id}, where $X = X_0 \times \cdots \times X_{e+f-1} \cong Q^{e+f}$, $X_r \cong Q$ for each $r \in [e+f]$, and $Y \cong \sym(e+f)$. Then $G \leqslant W$ and the following hold:
	\begin{enumerate}
	\item \label{rho_tau} $P^\rho \cap X  = X_0 \times \cdots \times X_{e-1}$ and $P^\rho \cap Y = T^\rho\leqslant \sym(e)\times 1 < \sym(e+f)$;
	\item \label{sigma-Y} the shuffle $\sigma$ lies in $Y$;
	\item \label{G cap Y} $G \cap Y = \langle \rho_\tau, \sigma \mid \tau \in T \rangle = \langle T^\rho, \sigma \rangle$;
	\item \label{X sub G} if $T$ is transitive on $[e]$, then $G \cap Y$ is a transitive subgroup of $Y$, $X \leqslant G$, and consequently $G = X \rtimes (G \cap Y) = Q \wr (G \cap Y)$ in product action on $[\ell]^{e+f}$;
	\item \label{e div f} if $f=ce$ for some $c \in \mathbb N$, then $G \cap Y \cong T \wr C_{1+c}$ (acting imprimitively on $[e+f]$);
	\item \label{e ndiv f} if $f \nmid e$ and $T = \sym(e)$, then $G \cap Y = \sym(e+f)$.
	\end{enumerate}
\end{prop}

\begin{proof}
By Lemmas \ref{lem:pa-shuffle} and \ref{lem:pa-P}, $T^\rho$ and $\sigma$ both lie in the top group $\sym(e+f)$ of $W$, and we comment on this in more detail. By Lemma \ref{lem:pa-P}, each element $(x_0, \dots, x_{e-1})$ of $S$ is mapped by $\rho$ to the element $(x_0, \dots, x_{e-1}, 1, \dots, 1)$ of $X$, and each element of $T$ is  is mapped by $\rho$ to an element of $Y$ with support in $[e] = \{0, \dots, e-1\}$. Thus $S^\rho\cong S$ is the subgroup $X_0 \times \cdots \times X_{e-1}$ of $X$, and $T^\rho\cong T$ is the subgroup $\{ y \in Y \ | \ \textnormal{supp}(y) = [e] \}$ of $Y$. Therefore $P^\rho \cap X = S^\rho$ and $P^\rho \cap Y = T^\rho$, which proves (\ref{rho_tau}). By Lemma \ref{lem:pa-shuffle}, the action of $\sigma$ on $[kn]$ induces a permutation of the coordinates of points of $[\ell]^{e+f}$ which sends each $r \in [e+f]$ to $(r+e) \pmod{e+f}$, so $\sigma \in Y$. This proves (\ref{sigma-Y}).

Now $G = \langle P^\rho, \sigma \rangle$, $P^\rho = (P^\rho\cap X)\rtimes (P^\rho\cap Y)$ by part (\ref{rho_tau}),  and $\sigma\in Y$ by part \eqref{sigma-Y}. Thus $G \cap Y = \langle P^\rho, \sigma \rangle \cap Y = \langle P^\rho \cap Y, \sigma \rangle = \langle T^\rho, \sigma \rangle$, 
proving part (\ref{G cap Y}). Now suppose that $T$ is transitive on $[e]$. Let $r \in [e+f]$ and write $r = pe + q$ for some $p, q \in \mathbb{Z}$ with $0 \leqslant q \leqslant e-1$. By Lemma \ref{lem:pa-shuffle}, $\sigma^{p}$ maps an arbitrary card $i \sim (i_0, \ldots, i_{e+f-1})$ to $i^{\sigma^{p}} \sim (i_{pe}, \ldots, i_{e+f-1}, i_0, \ldots, i_{pe-1})$ so in $i^{\sigma^{p}}$ the entry $i_r$ appears in position $q \in [e+f]$. Then since $T$ is transitive on $[e]$ and $q<e$, some element $\tau \in T$ satisfies $q^\tau = 0$, and hence by Lemma \ref{lem:pa-P}, $\rho_\tau$ permutes the entries of the tuple $i^{\sigma^{p}}$ in such a way that $i^{\sigma^{p}\rho_\tau}$ has entry $i_r$ in position $0$. Thus $\langle T^\rho, \sigma \rangle$ is transitive on $[e+f]$, proving the first part of (\ref{X sub G}).

To complete the proof of (\ref{X sub G}), note that $G \cap X \geqslant P^\rho \cap X = S^\rho = X_0 \times \cdots \times X_{e-1}$ by part (\ref{rho_tau}), and we have just proved that $G \cap Y$ is transitive on $[e+f]$. Now the action of $G \cap Y$ on $[e+f]$ is permutationally isomorphic to the conjugation action of $G \cap Y$ on the set $\{X_0, \ldots, X_{e+f-1}\}$ of direct factors of $X$. It follows that each $X_j \leqslant G \cap X$ and hence $X \leqslant G$. Thus $G = X \rtimes (G \cap Y) = Q \wr (G \cap Y)$ and (\ref{X sub G}) is proved.

For (\ref{e div f}) and (\ref{e ndiv f}) set $E := T^\rho$, and for any integer $r$ set $E_r := E^{\sigma^{-r}}$ and $\Delta_r := \textnormal{supp}(E_r) = \{ re, re+1, \dots, (r+1)e-1 \}$, reading entries of $\Delta_r$ modulo $e+f$.

Assume first that $f$ is a multiple of $e$. By Lemma \ref{lem:pa-shuffle}, $\sigma$ sends each  $r \in [e+f]$ to $(r-e)\pmod{e+f}$, and $\sigma$ has order $(e+f)/e = 1 + f/e$. For distinct $r, s \in [1+f/e]$ we see that $\Delta_r \cap \Delta_s = \varnothing$, so $[E_r,E_s] = 1$. In fact $\prod_{r=0}^{f/e} E_r \cong T^{1+f/e}$, and this group is normalised by $\sigma$. Clearly, $G \cap Y = \langle E, \sigma \rangle = \langle E_r, \sigma \ | \ 0 \leqslant r \leqslant f/e \rangle \cong T \wr C_{1+f/e}$. Hence (\ref{e div f}) is proved.

Now assume that $f$ is not a multiple of $e$, and that $T = \sym(e)$. Then $E_r = \sym(\Delta_r)$ for each $r \in [e+f]$. Write $e+f = pe + q$ for some $p, q \in \mathbb{Z}$ with $0 \leqslant q \leqslant e-1$, and note that $q \geqslant 1$ since $f$ is not a multiple of $e$. Now $[e+f] = \bigcup_{r=0}^p \Delta_r$. Note that if, for finite sets $\Delta$ and $\Delta'$ we have $\Delta \cap \Delta' \neq \varnothing$, then $\langle \sym(\Delta), \sym(\Delta') \rangle = \sym(\Delta \cup \Delta')$. For $0 \leqslant j \leqslant p-1$, we have $\Delta_{p+j+1} = \{ (p+j+1)e, \ldots, (p+j+2)e-1 \}$  (modulo $e+f$),  which is 
$\{ (j+1)e-q, \ldots, (j+2)e-q-1 \}$ (modulo $e+f$) since  $e+f = pe+q$, and so $\Delta_{p+j+1} \cap \Delta_j$ contains $(j+1)e-1$ (since $q \geqslant 1$) and $\Delta_{p+j+1} \cap \Delta_{j+1}$ contains $(j+1)e$ (since $q \leqslant e-1$). Thus (applying our observation twice), $\langle E_j, E_{j+1}, E_{p+j+1} \rangle = \sym(\Delta_j \cup \Delta_{j+1} \cup \Delta_{p+j+1})$ and it follows that $\langle E_j \ | \ 0 \leqslant j \leqslant p \rangle = \sym\big(\bigcup_{j=0}^p \Delta_j\big) = \sym([e+f]) \cong \sym(e+f)$. Since each $E_j \leqslant G \cap Y$ we conclude that $G \cap Y \cong \sym(e+f)$, proving (\ref{e ndiv f}).
\end{proof}

If $\ell$ is prime, then the identification in \eqref{eq:product-id} also defines an identification of $[\ell^m]$ with the vector space $\mathbb F_\ell^m$, relative to a given basis of $\mathbb F_\ell^m$.


\begin{prop} \label{prop:affine}
Assume we are in the power case with $k=\ell^e$, $n=\ell^f$ and that $\ell$ is prime. Suppose that $P =  V \rtimes T \leqslant \AGL(V) = \AGL(e,\ell)$ is an affine subgroup of $\sym(k)$ preserving the identification of $[k]$ with $\mathbb F_\ell^e$ as  in \eqref{eq:product-id}. Then there is an identification of $[kn]$ with $W:=\mathbb F_p^{e+f}$ such that $\sh(P,\ell^f) \leqslant \AGL(W)$. Further,
\begin{enumerate}
\item $\sh(P,n) = W \rtimes S$, where $S \leqslant \GL(W)$ and $S=\langle T^\rho,\sigma\rangle$,
\item if $f=ce$ for some $c\in \mathbb N$, then $S  \cong T \wr C_{c+1}$, an imprimitive linear group on $W$,
\item if $e \nmid f$ and $T= \GL(V)$, then $S  = \GL(W)$.
\end{enumerate}
\end{prop}

\begin{proof}
Recall from (\ref{eq:base-l}) and (\ref{eq:product-id}) that each $i \in [kn]$ can identified with the $(e+f)$-tuple $(i_0, \ldots, i_{e+f-1})$ of coefficients of its base-$\ell$ expansion $i = \sum_{r=0}^{e+f-1} i_r\, \ell^{e+f-1-r}$, where $0 \leqslant i_r \leqslant \ell-1$ for each $r$. This establishes a bijection of $[kn]$ with $W$ and this particular bijection defines a subgroup $\AGL(W)$ of $\sym(kn)$. Similarly,  $P$ respects the identification of $[k]$ with $V=\mathbb F_\ell^e$ as in   \eqref{eq:product-id}. We can view each $i \in [kn]$ as a vector $v+u \in W$ where $v = (v_0, \dots,v_{e+f-1})$ with $v_r=i_r$ for $r=0,\dots,e-1$ and $v_r=0$ for $r=e,\dots,e+f-1$ and $u = (u_0, \dots, u_{e+f-1})$ with $u_r = 0$ for $r=0,\dots,e-1$ and $u_r=i_r$ for $r=e,\dots,e+f-1$. Also
	\[ 
	i = \sum_{r=0}^{e-1} i_r\, \ell^{e+f-r-1} + \sum_{r=e}^{e+f-1} i_r\, \ell^{e+f-r-1} = 
	\ell^f \sum_{r=0}^{e-1} i_r\, \ell^{e-r-1} + \sum_{r=e}^{e+f-1} i_r\, \ell^{e+f-r-1} = an + b 
	\]
where $a = \sum_{r=0}^{e-1} i_r\, \ell^{e-r-1}\in[\ell^e]$ and $b = \sum_{r=0}^{f-1} i_{e+r}\, \ell^{f-r-1} \in[\ell^f]$. Since $b < \ell^f = n$, we see that the first $e$ coordinates of the representation of $i$ determine the column $a$ that card $i$ lies in and the last $f$ coordinates determine the row $b$.

We work with the following bases: let $w_0 = (1,0,\dots,0)$, $w_1 = (0,1,0,\dots,0)$, \dots, $w_{e+f-1} = (0,\dots,0,1)$ be basis vectors for $W$, and let $v_0=(1,0,\dots,0)$, \dots, $v_{e-1}=(0,\dots,1)$  be basis vectors for  $V$. We write elements of $\GL(W), \GL(V)$ as matrices relative to these bases.
By Lemma~\ref{lem:pa-shuffle},   the shuffle  $\sigma$  preserves the above identification of $[kn]$ with $W$,  and permutes the coordinates of a card by moving each entry $e$ places to the left (reading subscripts modulo $e+f$). Thus  $\sigma$  fixes $0 \sim (0,\dots,0)$ and acts as a linear map on $W$ corresponding to a permutation matrix, so $\sigma \in \GL(W)$.

Let $\tau \in P$, and consider the action of $\rho_\tau$ on a typical element $i=an+b$, as above, so $a \sim (a_0,\dots,a_{e-1})\in V$, and $i=an+b \sim
(a_0,\dots,a_{e+f-1})\in W$.  Then   $i^{\rho_\tau} = (an+b)^{\rho_\tau} = a^{\tau}n+b$ by \eqref{eq: rho}.  
If $\tau=t_v$ is a translation by some vector $v=(v_0,\dots,v_{e-1})$ in $V$, then  $a^{\tau} \sim (a_0,\dots,a_{e-1})+(v_0,\dots,v_{e-1})$ and hence $i^{\rho_\tau} \sim (a_0,\dots,a_{e+f-1}) + (v_0,\dots,v_{e-1},0,\dots,0)$, so $\rho_\tau$ is a translation by a vector of $W$.
 On the other hand, if $\tau \in \GL(V)$ corresponds to multiplication by a matrix $A$, then $a^\tau = aA= (a_0',\dots,a_{e-1}')$, say, and so $i^{\rho_\tau} \sim (a_0',\dots,a_{e-1}',a_e,\dots,a_{e+f-1})$ is the image of $(a_0,\dots,a_{e+f-1})$ under the matrix $\left [ \begin{array}{cc} A & 0 \\ 0 & I_f \end{array}\right ]$ in $\GL(W)$. Hence $\rho_\tau \in \GL(W)$. Thus $S = \langle T^\rho,\sigma\rangle  = 
 \langle \rho_\tau, \sigma \mid \tau \in T \rangle \leqslant \GL(W)$.
 
From the previous paragraph we have that, for each $r\in [e]$, if $\tau=t_{v_r}$ is the translation by $v_r$, then $\rho_{\tau}=t_{w_r}$ is the translation by $w_r$. We claim that the span $ \langle (V^{\rho})^{\sigma^i} \mid i\in [e+f] \rangle = W$. 
 For an arbitrary $i \in [e+f]$, write $i = pe+q$ with $q\in [e]$ and $p\in [f]$. Then Lemma~\ref{lem:pa-shuffle} implies that the conjugate
 \[
 (\rho_{t_{v_q}})^{\sigma^{-p}}=(t_{w_q})^{\sigma^{-p}} = t_{w_{q+pe}}= t_{w_{i}}
 \]
and this lies in $(V^\rho)^{\sigma^{-p}}=(V^\rho)^{\sigma^{e+f-p}}$. Hence the claim is proved. Now $\sh(P,n) = \langle V^\rho, T^\rho, \sigma \rangle \leqslant W S$,  and since $W \leqslant \sh(P,n)$, we have equality. This establishes part (1).

If $f=ce$, then $\sigma$ has order $c+1$ and the $c+1$ conjugates of $T^\rho$ under $\sigma$ pairwise commute. Hence $S \cong T \wr C_{c+1}$ is an imprimitive linear group, preserving a decomposition $W=\bigoplus_{i=1}^{c+1} V_i$ where each $V_i$ has dimension $e$ and is in bijection with $(V^\rho)^{\sigma^{-i}}$. Thus (2) holds.

Suppose now that $e \nmid f$ and that $T=\GL(V)$. Similarly to the proof of Proposition~\ref{prop:pa}(6), we use the fact that if $W$ can be written as $X+Y$ with $X \cap Y \neq \{0\}$, then, for the obvious embeddings of $\GL(X)$ and $\GL(Y)$ in  $\GL(W)$, we have $\GL(W) = \langle \GL(X), \GL(Y) \rangle$.  Since $e\nmid f$, there are integers $p, q$ with $1 \leqslant q \leqslant  e-1 $ such that  $e+f=pe+q$. Thus $q \equiv -pe \pmod{e+f}$ and we set 
$$
\begin{array}{cll} 
W_r &:= \langle w_{re}, \dots, w_{(r+1)e-1} \rangle  & \text{for }  r=0,1,\dots, p-1,\\
W_r' &:= W_0^{\sigma^{r}} = \langle w_{(p-r)e+q},\dots,w_{(p-r+1)e+q-1} \rangle  & \text{for }  r=1,\dots,p.
\end{array}
$$
For $r=0,\dots, p-1$, $W_r \cap W_{p-r}' $ contains $w_{re+q}$ and hence is non-zero,  and also, if $r\leqslant p-2$, then  
$W_{r+1} \cap W_{p-r}' $ contains $w_{(r+1)e}$ and hence is non-zero.  Now
 $(T^\rho)^{\sigma^{-r}}$ and $(T^\rho)^{\sigma^{p-r}}$ are isomorphic to and induce the full general linear groups on $W_r$ and $W_{p-r}'$ respectively. Hence, repeatedly applying the fact mentioned above to $W_0+W_p'$, $W_0+W_p' + W_1$, $W_0 + W_p'+W_1+W_{p-1}'$, etc., 
 we obtain $S = \GL(W)$, proving part (3).
\end{proof}

\begin{proof}[Proof of Theorem \ref{thm:productid}]
Part (1) follows from setting $P= \sym(\ell) \wr \sym(e)$ in Proposition~\ref{prop:pa}. Part (2) follows from setting $P = \AGL(e, \ell)$ in Proposition~\ref{prop:affine}.
\end{proof}

\begin{proof}[Proof of Corollary~\ref{cor:productid}]
We assume that $f$ is not a multiple of $e$ and that $k\neq 4$. In particular, this means that $e \geqslant 2$ and $e+f \geqslant 3$. Without loss of generality, we may assume that $\ell$ is not itself a proper power.  If $\ell \geqslant 5$ set $P=\sym(\ell)\wr\sym(e)$  and otherwise (i.e. if $\ell=2$ or $3$), set $P = \AGL(e,\ell)$.  Now, let $G = \sh(\sym(k),n)$    and $H = \sh(P, n)$. By Theorem~\ref{thm:productid}, if $\ell \geqslant 5$, then $H = \sym(\ell)\wr\sym(e+f)$, otherwise $\ell=2$ or $3$ and  $H = \AGL(e+f,\ell)$. Since $P \leqslant\sym(k)$, we have $H \leqslant G$.

\smallskip\noindent
\emph{Claim 1. $H \leqslant \alt(kn)$ if and only if $\ell$ is even; and $G\leqslant \alt(kn)$ if and only if $\ell$ is even.} \quad
If $\ell$ is even, then both $k$ and $n$ are even, and since $e \geqslant 2$, $k \equiv 0 \pmod{4}$. Hence by Corollary~\ref{cor:parity}(1) and (2),
 $G \leqslant \alt(kn)$ and hence also $H \leqslant \alt(kn)$.
Suppose now that $\ell$ is odd. We will show that $H \nleqslant \alt(kn)$, and hence also that $G \nleqslant \alt(kn)$. Since $\ell$ is odd, both $k$ and $n$ are odd. Further, if $\ell \geqslant 5$ then $P=\sym(\ell)\wr\sym(e)\not\leqslant \alt(\ell^e)$, and also if $\ell=3$ then $P= \AGL(e,3) \nleqslant \alt(3^e)$. Hence by Corollary~\ref{cor:parity}(3), $H \nleqslant \alt(kn)$.

\smallskip\noindent
\emph{Claim 2. One of the following occurs: \ \ (i) $G = H$; \ \ (ii) $G = \sym(kn)$ ($\ell$ odd) or $\alt(kn)$ ($\ell$ even)}.
Let $X = \sym(kn)$ if $\ell$ is odd or $\alt(kn)$ if $\ell$ is even. By Claim 1, $H \leqslant G \leqslant X$. If $H$ is maximal in $X$ then $G$ is as in (i) or (ii), so assume that $H < G < X$. If $\ell \geqslant 5$ then since   $e+f \geqslant 3$, it follows from \cite[Theorem, and Remark 2]{LPSon} that $H$ is maximal in $X$, contradiction. Thus $\ell = 2$ or $3$ and $H=\AGL(e+f,\ell)$. Then, since $e \geqslant 2$ and $k\ne 4$, again \cite[Theorem]{LPSon} implies that $H$ is maximal in $X$, a contradiction.

It remains to prove that case (i) of Claim 2 does not arise.

\smallskip\noindent
\emph{Case (i) of Claim 2 does not arise if $\ell \geqslant 5$.} Suppose that $G = H$. Let $\tau \in \sym(k)$ be a transposition. Then $\rho_\tau \in G$ has support size $2n = 2\ell^f$. Thus $\textrm{mindeg}_{[\ell]^{e+f}}(G) \leqslant 2\ell^f$. On the other hand, by Lemma~\ref{lem:mindeg}, $\textrm{mindeg}_{[\ell]^{e+f}}(G) = 2\ell^{e+f-1}$, and hence $e = 1$, which is a contradiction.

\smallskip\noindent
\emph{Case (i) of Claim 2 does not arise if $\ell=2$ or $3$.} Suppose that $\ell = 2$ or $3$ and that $G=H$ so $G = \AGL(e+f,\ell)$. Note that $G$ is not regular and consider a non-identity element $g \in G$ with at least one fixed point. Without loss of generality we may assume that $g$ fixes the zero vector of the vector space so that $g$ lies in  $\GL(e+f,\ell)$. The set of fixed points of $g$ forms a subspace, and so $\big|\supp_{[\ell]^{e+f}}(g)\big| = \ell^{e+f} - \ell^j$, for some $j\geqslant0$. Now choose $g = \rho_\tau$ for some $\tau \in \sym(k) $ moving $u$ points of $[k]$ with $u\geqslant2$. Then $\big|\supp_{[\ell^{e+f}]}(g)\big| = un = u\ell^f$, so $\ell^{e+f}-\ell^j=u\ell^f$ and hence $j\geqslant f$ and $\ell^e - \ell^{j-f}=u$.  If we choose $\tau$ to be a transposition in $\sym(k)$ then $u = 2$, which is coprime to $3$. Hence if $\ell=3$ then $j$ must be equal to $f$ and $3^e - 1 = 2$, which 
implies that $e=1$, contradiction. Thus $\ell=2$ and we have $2^e - 2^{j-f}=2$. Since $e\geqslant2$ this implies that $j=f+1$ and $2^e=4$, a contradiction since $k\ne 4$.
\end{proof}



\section{Primitivity of $\sh(P,n)$ for $k \geqslant 3$} \label{sec:prim}

The situation where $k = 2$ was completely dealt with in \cite{DGK} (see Theorem~\ref{dgk}). A first observation is that each group $\sh(\sym(2),n)$ is  imprimitive. On the other hand it follows from the first assertion of Theorem~\ref{thm:prim} (which we prove in this section) that 
$\sh(\sym(k),n)$ is primitive for each $k\geqslant 3$. In particular, to prove Theorem~\ref{thm:prim}, we only need to establish its first assertion.

Recall that $[u] = \{0, 1, \dots, u-1\}$ for $u$ a positive integer. 
For $a \in [k]$ and $b \in [n]$, define
	\begin{align}
	\mathcal{C}_a &= \{ an + b' \ | \ b'\in [n] \} \label{eq:col} \\
	\mathcal{R}_b &= \{ a'n + b \ | \ a'\in [k] \}. \label{eq:row}
	\end{align}
Thus the sets $\mathcal{C}_a$ are the piles of cards in $ [kn]$, and for each $b$, the set $\mathcal{R}_b$ consists of the $b$th card from each pile.

Let $P$ be a transitive subgroup of $\sym(k)$ and set $G = \sh(P,n)$. Recall that $G$ is a permutation group on $  [kn]$, and that the action of its subgroup $P^\rho$ on $\{ \mathcal{C}_a \ | \ a\in [k] \}$ is permutationally isomorphic to the action of $P$ on $[k]$. 
Our next result uses the notion of an orbital digraph: for $a \in [k] \setminus\{0\}$, the $P$-orbit on ordered pairs containing $\Delta = (0,a)^P$ is called an orbital, and its associated orbital digraph $\mathcal{G}(\Delta)$ has vertex set $[k]$ and set of directed edges $\Delta$. By definition $P$ acts arc-transitively on $\mathcal{G}(\Delta)$.

\begin{proof}[Proof of Theorem \ref{thm:prim}]
Assume for the sake of contradiction that $G$ is imprimitive on $[kn]$ for some $k\geqslant3, n\geqslant2$, and let $B$ be a nontrivial block of imprimitivity for $G$ containing $0$.  Recall that $\sigma \in G_0$. Thus $\sigma$ fixes $B$ setwise.

\medskip\noindent
\emph{Claim 1. If $\tau \in P$ is such that $0^\tau \ne 0$ in $[k]$ and $B^{\rho_\tau} = B$ in $[kn]$, then $\mathcal{R}_b \subseteq B$ for all $b \in B \cap \mathcal{C}_0$.} \quad To prove this let $b \in B \cap \mathcal{C}_0$ and $r \in [k]$, and set $a := 0^\tau$, and consider the orbital $\Delta := (0,a)^P$ in $[k]\times [k]$. Since $P$ is primitive on $[k]$, the orbital digraph $\mathcal{G}(\Delta)$ for $\Delta$ is connected, see \cite[Theorem 3.2A]{dixmort}. We shall prove that $rn + b \in B$ by induction on the distance $d = d_\Delta(0,r)$ from $0$ to $r$ in $\mathcal{G}(\Delta)$. This is true if $d = 0$ by assumption. Suppose then that $d \geqslant 1$, and that $sn + b \in B$ whenever $d_\Delta(0,s) \leqslant d-1$. Then there exists $t \in [k]$ such that $d_\Delta(0, t) = d-1$ and $d_\Delta(t, r) = 1$. Now  both $(0,a)$ and $(t,r)$ are arcs of $\mathcal{G}(\Delta)$, so for some $\delta \in P$ we have $0^\delta = t$ and $a^\delta = r$. By (\ref{eq: rho}), the image of $b \in B$ under $\rho_\delta$ is equal to $b^{\rho_\delta} = (0.n+b)^{\rho_\delta} = 0^\delta n + b = tn + b$, and since $tn + b \in B$ by induction, it follows that $B \cap B^{\rho_\delta}$ contains $tn + b$ and hence $\rho_\delta$ fixes $B$ setwise. Thus also $\rho_\tau\rho_\delta = \rho_{\tau\delta}$ fixes $B$ setwise, and hence $B$ contains $b^{\rho_{\tau\delta}}$. Again using (\ref{eq: rho}), 
	\[ B \ni b^{\rho_{\tau\delta}} = 0^{\tau\delta} n + b = a^\delta n + b = rn + b. \]
Thus, by induction, $B$ contains $rn+b$ for all $r \in [k]$, proving the claim.
  	
\medskip\noindent
\emph{Claim 2. If $\tau \in P$ is such that $0^\tau \ne 0$ in $[k]$, then $B \cap B^{\rho_\tau} = \varnothing$.} \quad Assume to the contrary that $0^\tau \neq 0$ in $[k]$, and $B^{\rho_\tau} = B$. Then $\mathcal{R}_0 \subseteq B$ by Claim 1. Since $\sigma$ fixes $B$ setwise, we have  $\mathcal{R}_0^\sigma  \subseteq B$, and hence, using \eqref{eq: s}, $\mathcal{R}_0^\sigma = \{0, \ldots, k-1\}  \subseteq B$. Setting $m_1 = \min\{n-1, k-1\}$ it follows again from Claim 1 that $\bigcup_{b=0}^{m_1} \mathcal{R}_b \subseteq B$. For $r \geqslant 1$, set $m_r := \min\{n-1,\,k^r-1\}$. Let $r$ be maximal such that $\bigcup_{b=0}^{m_r} \mathcal{R}_b \subseteq B$, and note that we have shown that this holds for $r=1$. If $m_r = n-1$ then $B$ contains $\bigcup_{b=0}^{n-1} \mathcal{R}_b =[kn]$, contradicting the fact that $B$ is nontrivial. Thus $m_r = k^r-1 < n-1$, and again using the fact that $\sigma$ fixes $B$ setwise, it follows that, for all $b$ satisfying $0 \leqslant b \leqslant k^r - 1$,  
	\[ 
	B \supseteq \mathcal R_b^\sigma = \left\{ (an+b)^\sigma \ | \ a \in [k] \right\} =	\left\{ bk+a \ | \ a \in [k] \right\}. 
	\]
In particular, $\left\{ 0, \ldots, k^{r+1}-1 \right\} \subseteq B$. If $k^{r+1}-1 \leqslant n-1$ then applying Claim 1, $\bigcup_{b=0}^{m_{r+1}} \mathcal R_b \subseteq B$, contradicting the maximality of $r$. Hence $k^{r+1}-1 > n-1$, so $m_{r+1} = n-1$, $B$ contains $\mathcal{C}_0$, and  applying Claim 1, $[kn] = \bigcup_{b=0}^{n-1} \mathcal R_b \subseteq B$, a contradiction.

\medskip\noindent
\emph{Claim 3. $B \subseteq \mathcal{C}_0$.} \quad Indeed, let $a \in [k] \setminus \{0\}$. Since $P$ is primitive and non-regular the only point fixed by $P_a$ is $a$, so some $\tau \in P_a$ does not fix $0$. Thus $B^{\rho_\tau} \cap B = \varnothing$ by Claim 2. This means that $\rho_\tau$ must move every point of $B$, and hence $B \subseteq \supp(\rho_\tau) = \bigcup_{r \in \supp(\tau)} \mathcal{C}_r$. Thus, since $\tau$ fixes $a$, we have  $B \cap \mathcal{C}_a = \varnothing$, and since this holds for all $a\in [k]\setminus\{0\}$, 
it follows that $B \subseteq \mathcal{C}_0$, proving the claim. 

To obtain a final contradiction, let $b$ be the maximum element of $B$. Since $B$ is nontrivial, and using Claim 3, we have $0 < b \leqslant n-1$. Now since $\sigma$ fixes $B$ setwise, $B$ contains $b^\sigma$, and by \eqref{eq: s}, $b^\sigma  = bk$. Note that $bk\leqslant kn-1$ since $b\leqslant n-1$, so $b^\sigma
=bk\in [kn]$. However $bk > b$ (since $b \ne 0$), contradicting the maximality of $b$. Thus no nontrivial blocks exist and therefore $G$ must be primitive.
\end{proof}

\section{$2$-transitivity of $\sh(P,n)$ when $k > n$} \label{sec:2trans}

Assume throughout this section that $k > n \geqslant 2$. For any $a \in [k]$ and $b \in [n]$ let $\mathcal{C}_a$ and $\mathcal{R}_b$ be as in \eqref{eq:col} and \eqref{eq:row}, respectively. In addition let
	\begin{align*}
	\mathcal{C}'_a &= \mathcal{C}_a \setminus \{an\} = \{ an + b \ | \ b \in [n] \} \\
	\mathcal{R}'_b &= \mathcal{R}_b \setminus \{b\} = \{ an + b \ | \ a \in [k] \}.
	\end{align*}

The next lemma proves the first assertion of Theorem~\ref{thm:2trans}.

\begin{lem} \label{lem: 2trans-k>n}
Suppose that $2 \leqslant n < k$. If $P$ is $2$-transitive then $\sh(P,n)$ is $2$-transitive.
\end{lem}

\begin{proof}
Let $G = \sh(P,n)$, and let $\Delta = n^{G_0}$, the $G_0$-orbit containing $n$. We will show that $\Delta = [kn] \setminus \{0\}$, 
from which it follows that $G$ is $2$-transitive..

First observe that if $\tau \in P_0$ (in the action of $P$ on $[k]$), then $\rho_\tau$ fixes $\mathcal{C}_0$ pointwise, and so $\rho_\tau$ fixes $0$. In particular, since $P$ is 2-transitive, for any $a \in [k] \setminus \{0\}$ there exists $\tau \in P_0$ such that $1^\tau = a$. Hence, by \eqref{eq: rho}, we have $(n+b)^{\rho_\tau} = 1^\tau n + b = an + b$, which shows that $\mathcal{R}'_b \subseteq (n+b)^{G_0}$ for each $b \in [n]$, and in particular that $\mathcal{R}_0'\subseteq \Delta$. Also, recall from Lemma \ref{lem:shuffle acts} that $\sigma \in G_0$, and that for any $b \in [n] \setminus \{0\}$ we have, since $b < k$, that $(bn)^\sigma = b$ by \eqref{eq: s}.  It follows since $\Delta$ contains $\mathcal{R}_0'$ that $\Delta$ also contains 
\[
(\mathcal{R}_0')^\sigma = \{   (bn)^\sigma  \mid 0 < b < k\} \supset  \{   (bn)^\sigma  \mid 0 < b < n\} =  \{   b  \mid 0 < b < n\} = \mathcal{C}_0'.
\]

Suppose that $k \geqslant 2n$. Then, for $n \leqslant i \leqslant 2n-1$, we have $i \in \mathcal{C}_1$ and, since $n < k$, $in \in \mathcal{R}'_0 \subseteq \Delta$. Again $(in)^\sigma \equiv ikn \pmod{kn-1} = i$ by Lemma \ref{lem:shuffle acts}, and it follows, as above, that $\mathcal{C}_1 \subseteq \Delta$. Hence, for $b\leqslant n-1$, $n+b \in \Delta$ so $(n+b)^{G_0} = \Delta$, and we showed above that $\mathcal{R}'_b \subseteq  (n+b)^{G_0}$. Thus $\Delta$ contains $\cup_{b\in [n]}\mathcal{R}_b'$ as well as $\mathcal{C}_0'$, so $\Delta = [kn] \setminus \{0\}$. Therefore $G$ is 2-transitive on $[kn]$.

Finally suppose that $n+1 \leqslant k < 2n$, so $k = n+q$, where $1 \leqslant q \leqslant n-1$. Let $b \in [n]$. If $1 \leqslant b \leqslant q-1$ then $(n+b)n \in \mathcal{R}_0$, and by Lemma \ref{lem:shuffle acts}, $((n+b)n)^\sigma = n+b$. Since $\mathcal{R}'_0 \subseteq\Delta$, it follows that $n+b \in \Delta$ for $1 \leqslant b \leqslant q-1$, and hence that $\mathcal{R}'_b \subseteq\Delta$ for $1 \leqslant b \leqslant q-1$. Suppose first that $q \geqslant 2$.
 Then this gives in particular $\mathcal{R}'_1 \subseteq \Delta$, and as $\mathcal{C}'_0 \subseteq \Delta$ we have $\mathcal{R}_1 \subseteq \Delta$. 
Let $0 \leqslant a \leqslant n-q-1$ and set $b = q+a$.  Then $an+1 \in \mathcal{R}_1\subset \Delta$ so that $(an+1)^\sigma\in\Delta$, 
and by  \eqref{eq: s}, $(an+1)^\sigma =1\cdot k + a = n+q+a =  n+b$, so $n+b\in\Delta$.  Now $q \leqslant b \leqslant n-1$ so $n+b\in \mathcal{R}_b'$,
and hence $\mathcal{R}'_b \subseteq (n+b)^{G_0} = \Delta$. Since this is true for each $b=q,\dots, n-1$, it follows as in the previous case that
 $\Delta = [kn] \setminus \{0\}$, and $G$ is 2-transitive on $[kn]$.

This leaves the case where $k = n+1$. Let $\Delta':=(n+1)^{G_0}$. We have already shown that $\mathcal{C}'_0 \cup \mathcal{R}'_0 \subseteq \Delta$. 

\medskip\noindent
\emph{Claim. $\bigcup_{i=1}^{n-1} \mathcal{R}'_i \subseteq \Delta'$.} Let $1 \leqslant b \leqslant n-1$, so $b \in [k]$ and $n+b \in \mathcal{R}'_b \subseteq (n+b)^{G_0}$. Using \eqref{eq: s} we have $(n+b)^\sigma = bk+1 = bn + (b+1)$. Note that $b+1 \leqslant n < k$ so $b+1 \in [k]$, and hence $(n+b)^\sigma \in \mathcal{R}'_{b+1}$. Applying this for each $b \in \{1, \ldots, n-1\}$ shows that each $\mathcal{R}'_{b+1}$ is contained in $(\mathcal{R}'_1)^{G_0} \subseteq (n+1)^{G_0}=\Delta'$, proving the claim. 

Thus the union $\Delta\cup \Delta'=[kn]\setminus\{0\}$, and to complete the proof we simply need to find a single point in $\Delta\cap \Delta'$. Consider $1\in \mathcal{C}_0'\subset \Delta$. Then $\Delta$ also contains  $1^\sigma = k = n+1 \in \mathcal{R}'_b \subset \Delta'$. Thus the orbits $\Delta, \Delta'$ intersect nontrivially, and hence are equal. Thus $\Delta = [kn] \setminus \{0\}$, completing the proof.
\end{proof}

\begin{lem}[Theorem \ref{thm:2trans} \eqref{2tr-alt}] \label{lem:2trans-full}
Suppose that $2 \leqslant n < k$ and let $P = \sym(k)$ or $\alt(k)$. Then either $(k,n) = (4,2)$ and $\sh(P,n) = \AGL(3,2)$, or $(k,n) \neq (4,2)$ and $\sh(P,n) = \sym(kn)$ or $\alt(kn)$.  Moreover, $\sh(P,n) = \alt(kn)$ if and only if one of the following holds:
	\begin{enumerate}
	\item $n_{(4)} = 0$, or
	\item $n_{(4)} = 2$ and $k_{(4)} \in \{0,1\}$, or
	\item $n$ is odd, $P \leqslant \alt(k)$, and either $n_{(4)} = 1$ or $k_{(4)} \in \{0,1\}$;
	\end{enumerate}
and $\sh(P,n) = \sym(kn)$ if and only if one of the following holds:
	\begin{enumerate}
	\item[(4)] $n_{(4)}, k_{(4)} \in \{2,3\}$, or
	\item[(5)] $n$ is odd and $P = \sym(k)$.
\end{enumerate}
\end{lem}

\begin{proof}
Let $G = \sh(P,n)$ where $P = \sym(k)$ or $\alt(k)$. If $(k,n)=(4,2)$ then it was found that $G=\AGL(3,2)$ in \cite[Table 2]{MM}. 
So we may assume that $(k,n) \neq (4,2)$. Suppose for a contradiction that $G \ngeqslant \alt(kn)$. Pick $1 \neq \tau \in P$ of minimal support $\mu(P)$ on $[k]$. Then the element $\rho_\tau$ of $G$ moves $\mu(P)n$ cards, and so
$\mu(G) \leqslant \mu(P)n$.
 Since $G$ is 2-transitive and doesn't contain the alternating group, it follows that $\mu(G) \geqslant (kn/4) - 1$ by a result of Bochert \cite{bochert} which can also be found with more details in de S\'eguier's book \cite[pp. 52--54]{deS}. Using the fact that $\mu(G) \leqslant \mu(P)n$ and rearranging, we have $k \leqslant 4\mu(P) + 4/n$. Now $\mu(P)=2$ or $3$ and since $n \geqslant 2$, we have $k \leqslant 14$. Since $2 \leqslant n < k \leqslant 14$, these cases may be checked by computer. We used \textsc{Magma} \cite{magma} which showed that $G = \alt(kn)$ or $G = \sym(kn)$ in the respective cases, which gives a contradiction. Hence $G$ contains $\alt(kn)$.  The conditions follow from Corollary \ref{cor:parity}.
\end{proof}

Recall that Burnside's Theorem states that a 2-transitive group is either of affine or almost simple type. Thus, both $P$ and the 2-transitive group $\sh(P,n)$ appearing in Lemma~\ref{lem: 2trans-k>n} have one of these types. In the remainder of this section we explore the connection between the type of $P$ and the type of $\sh(P,n)$. 

\begin{lem}[Theorem \ref{thm:2trans} \eqref{2tr-P-AS}] \label{lem:2trans-aff}
Suppose that $2 \leqslant n < k$, and suppose that $P$ is $2$-transitive on the set of piles. If $\sh(P,n)$ is affine then $P$ is affine.
\end{lem}

\begin{proof}
Suppose that $\sh(P,n)$ is affine, so that there is a prime $p$ and integers $e$ and $f$ such that $k = p^e$ and $n = p^f$, with $e > f$ since $k > n$. Let us assume for a contradiction that $P$ is almost simple, and let $T = \mathrm{soc}(P)$. So $T$ is a nonabelian finite simple group; moreover, $T$ is a transitive subgroup of $\sym(k$) since $T$ is normal in $P$ and $P$ is 2-transitive. Hence, for any stabiliser $H$ in $T$ (in the action of $T$ on $[k]$), we see that $|T:H| = k = p^e$. Thus, by \cite[Theorem 1]{guralnickppower}, $T$ and $H$ belong to one of five cases, which we consider separately below.

\medskip\noindent
\emph{Case (a). $T = \alt(k)$ and $H \cong \alt(k-1)$, where $k \geqslant 5$.} Then $P = \alt(k)$ or $\sym(k)$. Since $k \geqslant 5$, $\sh(P,n) = \alt(kn)$ or $\sym(kn)$ by Lemma \ref{lem:2trans-full}. However $kn \geqslant 5$, so neither $\alt(kn)$ nor $\sym(kn)$ is affine, a contradiction.
 
\medskip\noindent
\emph{Cases (c) and (d). $T = \PSL(2,11)$ and $H \cong \alt(5)$; or $T = M_{23}$ and $H \cong M_{22}$; or $T = M_{11}$ and $H \cong M_{10}$.} In these cases the degree of $T$ is prime, and so $k = p \in \{11,23\}$. Since $p^f = n < k = p$, this gives a contradiction.

\medskip\noindent
\emph{Case (e). $T = \PSU(4,2)$ and $|T:H| = 27$.} Since $k = 27$ either $n=3$ or $n=9$. Using \textsc{Magma} \cite{magma} we see that $\sh(P,n)$ contains $\alt(kn)$, a contradiction.

\medskip\noindent
\emph{Case (b). $T = \mathrm{PSL}(d,q)$ for some prime power $q$ and prime $d$, and $H$ is the stabiliser of a line or hyperplane.} In this case $k\geqslant 5$ and
	\begin{equation} \label{eq:p and q}
	p^e = k = |T:H| = \frac{q^d-1}{q-1}.
	\end{equation}
Let $V = \soc(\sh(P,n))$, so that $V$ is elementary abelian of order $p^{e+f}$. Since $T \leqslant \sh(P,n)$, and $T$ is simple, we have $T \cap V = 1$. Hence $T$ embeds into $\sh(P,n)/V$, which can be identified with a subgroup of $\mathrm{GL}(e+f,p)$. By \eqref{eq:p and q} we must have $p \neq q$. Thus we have a coprime representation of $T$, and so by \cite[Theorem 5.3.9]{KL} we have 
	\[ 
	e+f \geqslant \begin{cases} q^{d-1} - 1 &  \text{if $d \geqslant 3$}, \\ \dfrac{q-1}{(2,q-1)} & \text{if $d = 2$}.\end{cases} 
	\]
Suppose first that $d \geqslant 3$. If $q \geqslant 3$ then
	\[ 
	q^{d-1} - 1  \geqslant \frac{1}{2} \left( \frac{q^d - 1}{q-1} \right) = \frac{p^e}{2}, 
	\]
and so $p^e \leqslant 2e + 2f$. Recall that $e > f$; hence $p^e < 4e$ and since $k\geqslant5$ this implies that $p = 2$ and $k = p^e =8$. Now, since $q \geqslant 3$ and $d \geqslant 3$ we obtain from \eqref{eq:p and q} that
	\[ 
	k =  q^{d-1} + \dots + q + 1 \geqslant 3^2 + 3 + 1 > 8, 
	\]
a contradiction. If $q = 2$ then by \eqref{eq:p and q} we have 
	\[ 
	p^e = 2^d - 1 = 2\big( 2^{d-1} - 1 \big) + 1 \leqslant 2(e+f) + 1 < 4e 
	\]
also a contradiction. Hence  $d = 2$, and by \eqref{eq:p and q} we have $p^e = q+1$, so that $p^e = (q-1) + 2 \leqslant 2(e+f) + 2 \leqslant 4e$. Since
$k\geqslant5$, this implies that $e \geqslant 2$, and hence $(p,e) \in \{ (2,3), (2,4), (3,2) \}$. This gives $k = p^e = 8,16,9$. Since $k = q+1$, and $q$ is a prime power, we have $(k,q) \in \{(8,7), (9,8)\}$. Thus either $k = 8$, $n =2$ or $4$, and $T = \PSL(2,7)$, or $k = 9$, $n = 3$, and $T = \PSL(2,8)$. In these cases, we use {\sc Magma} to verify that $\sh(T,n)$ contains $\alt(kn)$, and since $\sh(T,n) \leqslant \sh(P,n)$, this contradicts $\sh(P,n)$ being affine. This completes the proof.  
\end{proof}

\subsection{The case where $P$ is affine and $\sh(P,n)$ is almost simple}

It follows from Lemma \ref{lem:2trans-aff} that the $2$-transitive group $\sh(P,n)$ is almost simple whenever $P$ is $2$-transitive and almost simple. If $P$ is affine, then by Proposition~\ref{prop:affine} the group $\sh(P,n)$ is affine whenever the power case holds, that is, whenever $k$ and $ n$ are powers of the same prime. We next consider the case where $P$ is $2$-transitive and affine but the power case does not hold, and determine which $2$-transitive almost simple group could arise as $\sh(P,n)$. Let us fix the following notation for the remainder of this section:
\begin{itemize}
\item $2 \leqslant n < k$ and $k=p^e$ for some prime $p$;
\item $P$ is a 2-transitive affine group of degree $k$;
\item $G := \sh(P,n)$ and $T:=\mathrm{soc}(G)$ is a finite nonabelian simple group.
\end{itemize}
We observe that  $kn$ is not prime, and so
	\begin{equation} \label{eq:kn}
	6 \leqslant p^e n = kn = m < k^2 = p^{2e},
	\end{equation}
and since $P$ is $2$-transitive, $|P|$ is divisible by $p^e(p^e-1)$. Further, by Lemma~\ref{lem: 2trans-k>n}, $G$ is 2-transitive. 

The classification of $2$-transitive almost simple groups is complete and the list of possible groups can be found in \cite[Table 7.4]{cameronpermgroups}; it consists of seven infinite families, which are listed in Table \ref{table:2transAS}, and several sporadic examples which we can deal with immediately.

\begin{table}[ht]
\begin{center}
\begin{tabular}{rllll}
 & $T=\soc(G)$	& degree $m$ of $G$ & Conditions & $|G/T|\leqslant$ \\
\hline
\small\sf 1 & $\Sp(2d,2)$	& $2^{d-1}(2^d + 1)$		& $d \geqslant 3$								& $1$ \\
\small\sf 2 & $\Sp(2d,2)$	& $2^{d-1}(2^d - 1)$		& $d \geqslant 3$								& $1$ \\
\small\sf 3	& $\Sz(q)$		& $q^2+1$								& $q=2^f$, $f \geqslant 3$ odd	& $f$	\\
\small\sf 4	& $\Ree(q)$		& $q^3+1$								& $q=3^f$, $f$ odd							& $f$	\\
\small\sf 5	& $\PSU(3,q)$	& $q^3+1$								& $q=q_0^f$, $q_0$ prime				& $(3,q+1)2f$ \\
\small\sf 6	& $\PSL(d,q)$	& $(q^d-1)/(q-1)$				& $q=q_0^f$, $q_0$ prime				& $(d,q-1)f$ 	\\
\small\sf 7 & $\alt(m)$ 	& $m$										& $m \geqslant 5$								& $2$
\end{tabular}
\end{center}
\caption{Infinite families of almost simple $2$-transitive groups $G$.}
\label{table:2transAS}
\end{table}

\begin{lem}
\label{lem:must be in table}
If $G$  is almost simple, then $T$ must appear in Table~$\ref{table:2transAS}$.
\end{lem}
\begin{proof}
If $G$ is almost simple and $T=\soc(G)$ does not appear in Table~\ref{table:2transAS} then $G$ must appear in one of rows 8--18 of \cite[Table 7.4]{cameronpermgroups}. Since $k = p^e$, $n < k$, and $kn$ is the degree of the group $\sh(P,n)$, clearly $kn$ is not prime. In each case, there is a unique choice for $k$ which in turn defines $n$, and we are then able to use {\sc Magma} \cite{magma} to calculate $\sh(P,n)$. We obtain a contradiction in each case. For example, if $\sh(P,n)= Co_3$ with  $kn= 276$ then $k = 23$, $n=12$ and $P$ is an affine group of degree $k$. The computer calculation shows that $\sh(P,n)$ contains $\alt(276)$, a contradiction.
\end{proof}

In Lemmas \ref{lem:2-trans Sp}--\ref{lem: 2-trans PSL} we consider the groups in rows 1--6 of Table~\ref{table:2transAS} and eliminate all but the short list given in Table~\ref{table:2transAS-ii} below. These remaining groups are treated in Theorem~\ref{thm:finish 2-trans}. First we show that $P^\rho$ intersects $T$ nontrivially.

\begin{lem} \label{lemma:socP}
If one of the rows of Table~$\ref{table:2transAS}$ holds for $G = \sh(P,n)$ and $T = \soc(G)$, then $P^\rho \cap T \geqslant \soc(P^\rho) = C_p^e$.
\end{lem}
\begin{proof}
This is trivially satisfied if row 1 or 2 of the table holds, since in this case $G = T$.

Suppose for the sake of contradiction that $P^\rho \cap T \ngeqslant \soc(P^\rho)$. Since $\soc(P)$ is the unique minimal normal subgroup of $P $ and $P \cong P^\rho$, we then have $P^\rho \cap T = 1$. Hence $P^\rho \cong P^\rho T/T \leqslant G/T$, and in particular $|G/T| \geqslant p^e(p^e-1)$.  Since $p^e \geqslant 3$, the 2-transitive group $P$ is not cyclic, and so $G/T$ is not cyclic. This rules out rows 3, 4, and 7. Consider next row 6. Now $G/T = C_r . C_s$ with $r$ dividing $(d,q-1)$ and $s$ dividing $f$.  Since the $2$-transitive group $P$ is nonabelian and isomorphic to a subgroup of $G/T$, we must have that $r \geqslant 3$, and thus $d\geqslant 3$.  Hence $q^2 \leqslant q^{d-1} < m < k^2 = p^{2e}$. On the other hand, we have $k=p^e > 2$, and since  $q=q_0^f$ we have $f \leqslant \frac{q}{2}$, so 
	\begin{equation} \label{eq:p^2e}
	\frac{p^{2e}}{2} < p^e(p^e-1) \leqslant |G/T| \leqslant (d,q-1)f \leqslant \frac{q^2}{2},
	\end{equation}
which gives $q^2 < p^{2e} < q^2$, a contradiction.

%

It remains to deal with row 5, and, since $T$ is nonabelian simple, we must have $q \geqslant 3$. 
Since $m=q^3+1\geqslant28$ it follows from \eqref{eq:kn} that $p^e\geqslant6$. Then, again using $f \leqslant q/2$ and \eqref{eq:kn},
	\[ 
	q^3 < m < p^{2e} < 2p^e(p^e-1) \leqslant 2|G/T| \leqslant 2(3,q+1)2f \leqslant 12f \leqslant 6q, 
	\]
which has no solutions with $q \geqslant 3$. Thus $P^\rho \cap T \geqslant \soc(P^\rho)$. 
\end{proof}

\begin{table}[ht]
\begin{center}
\begin{tabular}{llrr}
  & $T$ & $k$ & $n$ \\
\hline
\sf\small 1 & $\Sp(6,2)$ & $7$ & $4$ \\
\sf\small 2 & $\Sp(6,2)$ & $9$ & $4$ \\
\sf\small 3 & $\Sz(8)$ & $13$ & $5$ \\
\sf\small 4 & $\PSU(3,4)$ & $13$ & $5$ \\
\sf\small 5 & $\PSL(2,5)$ & $3$ & $2$ \\
\sf\small 6 & $\PSL(2,7)$ & $4$ & $2$ \\
\sf\small 7 & $\PSL(2,9)$ & $5$ & $2$ \\
\sf\small 8 & $\PSL(2,11)$ & $4$ & $3$ \\
\sf\small 9 & $\PSL(2,27)$ & $7$ & $4$ \\
\sf\small 10 & $\PSL(2,32)$ & $11$ & $3$ \\
\sf\small 11 & $\PSL(2,64)$ & $13$ & $5$ \\
\sf\small 12 & $\PSL(3,4)$ & $7$ & $3$ \\
\sf\small 13 & $\PSL(4,2)$ & $5$ & $3$ \\
\sf\small 14 & $\PSL(4,3)$ & $8$ & $5$ \\
\sf\small 15 & $\PSL(6,2)$ & $9$ & $7$ \\
\end{tabular}
\end{center}
\caption{Remaining $T = \soc(\sh(P,n))$ from rows 1--6 of Table~\ref{table:2transAS}, with $k$, and $n$ such that $P$ is affine $2$-transitive and $\sh(P,n)$ is almost simple}
\label{table:2transAS-ii}
\end{table}

\begin{lem}
\label{lem:2-trans Sp}
If $G = \sh(P,n) = \Sp(2d,2) = T$ as in row $1$ or $2$ of Table $\ref{table:2transAS}$, then row $1$ or $2$ of Table $\ref{table:2transAS-ii}$ holds.
\end{lem}

\begin{proof}
Let $G = \sh(P,n)$ and suppose for a contradiction that $G = \Sp(2d,2)$ for some integer $d \geqslant 3$. Then $kn = 2^{d-1}\big(2^d + \epsilon\big)$ where $\epsilon \in \{1,-1\}$. In particular, since $k > n$ and $k$ is a prime power, we have that $k$ divides $2^d + \epsilon$. By \eqref{eq:kn}, $2^{d-1}\big(2^d + \epsilon\big) < p^{2e}$. If $p^e = k \leqslant \big(2^d + \epsilon\big)/3$ then the above yields $\big(2^d + \epsilon\big)^2/9 > 2^{d-1}\big(2^d + \epsilon\big)$, so $2^d + \epsilon > 9 \cdot 2^{d-1} = 2^d + 7 \cdot 2^{d-1}$. Since $d \geqslant 3$, this gives a contradiction. Hence $p^e = k > \big(2^d + \epsilon\big)/3$, and since $k$ divides $2^d + \epsilon$ we have that
	\[ p^e = k = 2^d + \epsilon. \]
 If $d=3$ we get the two cases given in rows 1 and 2 of Table~\ref{table:2transAS-ii}, so we now assume $d\geqslant 4$. 

If $\epsilon = -1$ then $p^e = 2^d - 1$. In particular, $d\neq 6$ so by Lemma~\ref{lem:zsig}, $2^d-1$ has a primitive prime divisor, $r$ say. Since $r$ divides $2^d-1=p^e$, we have $r=p$. Let $c$ be a prime divisor of  $d$, then  $2^c-1$ divides $2^d-1=p^e$, and so $p$ must divide $2^c-1$. Since $p$ is a primitive prime divisor of $2^d-1$, it follows that $c=d$. Hence $d$ is prime, and since $d\geqslant 4$, $d$ is an odd prime.  
 If $\epsilon = 1$ then $p^e = 2^d + 1$.
By Lemma~\ref{lem:zsig} and since $d\geqslant 4$, $2^{2d} - 1$ has a primitive prime divisor, $r$ say, and $r$ does not divide $2^d-1$ by the definition of a primitive prime divisor. Thus $r$ must divide $2^d+1=p^e$, and again we have $r=p$. Hence, for both values of $\epsilon$, $p$ is a primitive prime divisor of $2^{2d}-1$.


We  claim that $k=p$ and $\big|N_G(\soc(P^\rho))/C_G(\soc(P^\rho))\big|$ divides $2d$. If $\epsilon = 1$    then since $p$ is a primitive prime divisor of $2^{2d}-1$, $\soc(P^\rho)$ is contained in a Singer cycle group $Z_{2^{2d}-1}$ of $\GL(2d,2)$ and $N_{\GL(2d,2)}(\soc(P^\rho)) \cong Z_{2^{2d}-1}.2d$ (see \cite[Kapitel II, 7.3]{HuppertI}). Thus $k=p$ and the claim is proved in this case. If $\epsilon = -1$ then  since $d$ is odd, by considering orders, we see that the stabiliser $G_U = 2^{\binom{d}{2}}.\GL(d,2)$ of a totally isotropic subspace $U$ of dimension $d$ contains a Sylow $p$-subgroup of $G$. Hence a Sylow $p$-subgroup of $G$ is isomorphic to a Sylow $p$-subgroup of $\GL(d,2)$. Since $p$ is a primitive prime divisor of $2^d-1$, the Sylow $p$-subgroups of $\GL(d,2)$, $G$, and therefore $\mathrm{soc}(P^\rho)$, are cyclic. Hence  $k=p$.  Since $\soc(P^\rho)$ preserves the decomposition $U \oplus U^*$ of the underlying vector space, and this decomposition is unique for $\soc(P^\rho)$, $N_G(\soc(P^\rho))$ also preserves this decomposition. Thus $N_G(\soc(P^\rho)) = N_{\GL(d,2).2}(\soc(P^\rho))$. The same argument as above now gives that $|N_G( \soc(P^\rho))/C_G(\soc(P^\rho))|$ divides $2d$, and hence the claim is proved for both cases.


It follows that $p - 1 \leqslant 2d$. This gives
	\[ 2dp \geqslant p(p-1) \geqslant kn = 2^{d-1}\big(2^d + \epsilon) = 2^{d-1}p \]
and so $2d \geqslant 2^{d-1}$. Since we have $d\geqslant 4$, this implies that $d = 4$. Since $2^d - 1 \leqslant 2^d + \epsilon = k = p \leqslant 2d + 1$, we get a contradiction if $d=4$. This completes the proof.
\end{proof}

\begin{lem}
\label{lem:2-trans sz}
If $T = \Sz(q)$ as in row $3$ of Table~$\ref{table:2transAS}$, then $G = T.3 = \aut(T)$ and row $3$ of Table $\ref{table:2transAS-ii}$ holds.
\end{lem}

\begin{proof}
Let $q = 2^f \geqslant 8$ with $f$ odd. Let $r = 2^{(f+1)/2}$ so $r^2 = 2q$. Then 
	\[ p^e n = m = q^2 + 1 = (q+r+1)(q-r+1) < p^{2e}. \]
Since $(q+r+1, q-r+1)=1$, $p^e$ divides one of these factors. If $p^e$ divides $q-r+1$ then $q^2+1 > p^{2e}$, a contradiction. 
Thus $p^e$ divides $q+r+1$, and we get a similar contradiction if $p^e \leqslant (q+r+1)/3$ since $3(q-r+1) > q+r+1$ for all $f \geqslant 3$. 
It follows (since $q+r+1$ is odd) that $p^e = q+r+1$, and $n = q-r+1$. Now all subgroups of $T = \Sz(q)$ of order $q+r+1$ are cyclic \cite[Section 4]{sz} 
 and hence, by Lemma~\ref{lemma:socP}, we must have $e = 1$ and $\AGL(1,p) \cong P^\rho\leqslant N_{\aut(T)}(\soc(P^\rho))$. By \cite{sz}, $|N_{\aut(T)}(C_p)| $ divides $ 4fp$, and hence $p-1 = q+r$ divides $4f$. The only solution is $f = 3$, $T = \Sz(8)$, in which case we do have $\AGL(1,13) \leqslant \aut(T)$ \cite[p.28]{atlas}. 
\end{proof}

\begin{lem}\label{lemma 2-trans Ree PSU}
If $T = \Ree(q)$ or $\PSU(3,q)$ with $q \geqslant 3$ as in row $4$ or $5$ of Table~$\ref{table:2transAS}$, then $G = \aut(T)$ and row $4$ of 
Table $\ref{table:2transAS-ii}$ holds.
\end{lem}

\begin{proof}
Here
	\[ p^e n = m = q^3 + 1 = (q+1)(q^2-q+1) < p^{2e}, \]
and $(q+1, q^2-q+1) = (q+1, 3)$.  If  $p^e$ divides $q+1$, then $n \geqslant q^2-q+1 \geqslant q+1 \geqslant p^e$, which is a contradiction. 
Thus either (a) $p \ne 3$ and $p^e$ divides $q^2-q+1$, or (b) $p = 3$ and $3$ divides $(q+1, q^2-q+1)$. 

Consider case (a). Here $p^e$ divides $q^6 - 1$. Let $c$ be minimal such that $p$ divides $q^c- 1$. Then $c$ divides $6$. If $c = 1$ then $p$ divides $(q^2-q+1, q-1) = 1$, a contradiction. If $c = 2$ then $p$ divides $(q^2-q+1, q^2-1) = (q-2,3)$, which contradicts the fact that $p \ne 3$. If $c = 3$ then $p$ divides $(q^3-1, q^3+1) = (q+1, 2)$, so $p = 2$, but then $p$ divides $q - 1$, contradicting $c = 3$. Thus $c = 6$ and $p$ is a primitive prime divisor of $q^6-1$.  Suppose that $T=\PSU(3,q)$. Now  a Sylow $p$-subgroup of $\mathrm{GU}(3,q)$ is contained in a Singer cycle subgroup of $\GL(3,q^2)$, and is therefore cyclic. Suppose now that $T = \mathrm{Ree}(q)$. The $p$-part of $|T|$ is equal to the $p$-part of $q^2-q+1$. Since $\mathrm{Ree}(q)$ contains cyclic subgroups of  order $q^2-q+1$ (see \cite{kleidman2g2}), the Sylow $p$-subgroups of $\mathrm{Ree}(q)$ are cyclic. Thus in both cases the Sylow $p$-subgroups of $T$ are cyclic, and hence $e = 1$ and $\AGL(1,p) \cong P^\rho \leqslant N_{\aut(T)}(\soc(P^\rho))$. This means that $p-1$ divides $|N_{\aut(T)}(\soc(P^\rho))/C_{\aut(T)}(\soc(P^\rho))| = 6f$  (see \cite[Table 8.5]{BHRD} if $T=\PSU(3,q)$ and \cite[Theorem C]{kleidman2g2} if $T=\mathrm{Ree}(q)$).  Since $np = q^3+1$ we have $q+1 \leqslant n < p \leqslant 6f+1$. The only values of $q = q_0^f$ satisfying $q < 6f$ are $q \in \{4,5,8,16\}$.  The case $q = 4$ gives row 4 of Table~\ref{table:2transAS-ii}. For $q = 5$, $8$, $16$ we have that $p-1$ divides $6f = 6$, $18$, $24$, and $p$ divides $q^2-q+1 = 21$, $57$, $241$ respectively. Since $p \ne 3$ (and $241$ is prime) this implies that $(q,p) = (5,7)$ or $(8, 19)$. These give a contradiction since in both cases $n = (q^3+1)/p > p$.

Now consider case (b) where $p = 3$ divides $(q+1, q^2-q+1)$. Thus $q \equiv 2\pmod{3}$, and since $q > 2$ we have $q \geqslant 5$. Also, since $3^{2e} > q^3+1 \geqslant 126$, we have $e \geqslant 3$. Now modulo $9$, $q \equiv 2, 5$ or $8$, and in each of these cases $q^2-q+1 \equiv 3\pmod{9}$.  Therefore $9$ does not divide $q^2-q+1$ and hence $3^{e-1} \geqslant 9$ divides $q+1$. Thus $q+1\geqslant 9$ and 
	\[ q^3+1 < 3^{2e} = (3^e)^2 \leqslant 9(q+1)^2 \]
and hence $q = 8$. Then $p^e = 27$, $n = 19$. However $\aut(T)$ does not contain a $2$-transitive group of degree $27$ since its order is not divisible by $13$. 
\end{proof}

\begin{lem}
\label{lem: 2-trans PSL}
If $T = \PSL(d,q)$ as in row $6$ of Table~$\ref{table:2transAS}$, then one of rows $5$--$15$ of Table $\ref{table:2transAS-ii}$ holds.
\end{lem}
\begin{proof}
In this case $m = (q^d - 1)/(q-1)$. Recall from \eqref{eq:kn} that $k$ divides $m$ and $m < k^2$. So $k \leqslant m_p < k^2$, and $p^e$ and $p$ both divide $q^d - 1$. Let $c$ be the smallest positive integer such that $p$ divides $q^c - 1$ (note that $c = 1$ if $p = 2$). Then $c$ divides $d$ by Lemma \ref{lemma:p-part}; let $\ell = d/c$. By Lemma~\ref{lemma:socP} we have $C^e_p = \mathrm{soc}(P^\rho) \leqslant T = \PSL(d,q)$. By \cite[Lemma 5.5.2]{KL} we have 
	\begin{equation} \label{eq:l>e}
	d \geqslant e.
	\end{equation}

\emph{Claim 1. If $p = 2$ then $d$ is even.} If $p = 2$ and $d$ is odd then by Lemma \ref{lemma:p-part} \eqref{p=2},
	\[ m_2 = \left(\frac{q^d-1}{q-1}\right)_2 = \frac{(q^d-1)_2}{(q-1)_2} = \frac{(q-1)_2}{(q-1)_2} = 1. \]
Since $k \leqslant m_2$ we have $k = 2^e \leqslant 1$, a contradiction. Therefore $d$ is even whenever $p = 2$.

Let $r, s \in \mathbb{Z}$ such that $p^r = (q^c-1)_p$ and $p^s = \ell_p = (d/c)_p$. Note that $r \geqslant 1$ since $p$ divides $q^c-1$. Also, for $p = 2$, let $t \in \mathbb{Z}$ such that $2^t = (q+1)_2$. Then by Lemma \ref{lemma:p-part} we have
	\begin{equation} \label{eq:p-part}
	m_p = \left(\frac{q^d-1}{q-1}\right)_p = \begin{cases} 2^{s+t-1} &\text{if $p = 2$}, \\ p^s &\text{if $p \neq 2$, $c = 1$}, \\ p^{r+s} &\text{if $p \neq 2$, $c > 1$.} \end{cases}
	\end{equation}

Now
	\[ m > q^{d-1} = q^{c\ell-1} \geqslant q^{c(\ell-1)} > (q^c-1)^{\ell-1}_p = p^{r(\ell-1)}, \]
and since $m < k^2 = p^{2e}$ it follows that
	\begin{equation} \label{eq:2e}
	2e > r(\ell-1).
	\end{equation}

\emph{Claim 2. If $p \neq 2$ then $c > 1$.} If $p \neq 2$ and $c = 1$ then, $m_p = p^s$, and since $k \leqslant m_p$ we have $e \leqslant s$ using \eqref{eq:p-part}, and thus $r(\ell-1) < 2s$ using \eqref{eq:2e}. So $p^s = \ell_p \leqslant \ell < (2s+r)/r$, which implies that $s \neq 0$. Hence 
$r < 2s/(p^s-1) \leqslant 1$ since $2s\leqslant 3^s-1\leqslant p^s-1$ for all $s\geqslant 1$ and $p\geqslant3$. Thus $r < 1$, a contradiction. 
Therefore $c > 1$ whenever $p \neq 2$.

\smallskip
By Claims 1 and 2 we only need to consider the cases where $p = 2$ and $d$ is even, and where $p \neq 2$ and $c > 1$.

\emph{Case 1.}  Assume that $p = 2$. Recall that in this case $c = 1$, so $d = \ell$.
 Since $3 \leqslant k = 2^e$ we must have $e \geqslant 2$, and using \eqref{eq:p-part} and $k \leqslant m_p$ we have $e \leqslant s+t-1$.
Suppose first that $d = 2$. Then by \eqref{eq:l>e} we have $e=2$ and so  $k = 4$. Thus
	\[ 
	8 \leqslant 4n = kn = m = \frac{q^2-1}{q-1} = q+1 < k^2 = 16, 
	\]
so $q = 7$ or $11$. If $q = 7$ then $k = 4$, $n = 2$, and $T = \PSL(2,7)$, as in row 6 of Table~\ref{table:2transAS-ii}. If $q = 11$ then $k = 4$, $n = 3$, and $T = \PSL(2,11)$, as in row 8 of Table~\ref{table:2transAS-ii}.
Suppose now that $d \geqslant 4$ (recall that $d$ is even). Then, using Lemma~\ref{lemma:p-part} and \eqref{eq:p-part}, we can write
$$ 
2^{2(s+t-1)}  = (m_2)^2 = \left ( \frac{(q^d-1)_2}{(q-1)_2}\right) ^2 = \left((q+1)_2\left(\frac{d}{2}\right)_2\right)^2 \leqslant \left((q+1) \cdot \frac{d}{2}\right)^2. 
$$
Since $m < k^2 = 2^{2e} \leqslant 2^{2(s+t-1)}$, this gives us $m < (q+1)^2d^2/4$ and hence
	\[ 
	d^2> \frac{4(q^d-1)}{(q-1)(q+1)^2} > \frac{4q^{d-2}}{q+1} > \frac{4q^{d-2}}{2q} \geqslant 2 \cdot 3^{d-3}. 
	\]
Hence $d \leqslant 5$, and therefore $d = 4$. So $2 \leqslant e \leqslant 4$  by \eqref{eq:l>e} and $k \in \{4,8,16\}$. From the fact that $m < k^2$ and $k$ divides $m$ we get that $k = 8$, $n = 5$, and $T = \PSL(4,3)$, as in row 14 of Table~\ref{table:2transAS-ii}.

\emph{Case 2.} Finally, assume that $p \neq 2$. Then $c > 1$ by Claim 2, so $m_p = p^{r+s}$ by \eqref{eq:p-part}. Since $p^e = k \leqslant m_p$ we have $e \leqslant r+s$, and using \eqref{eq:2e}, $r(\ell-1) < 2(r+s)$. Now  since $r \geqslant 1$, we have $2(r+s) \leqslant 2r(1+s)$. Hence $\ell-1 < 2(1+s)$, so that
	\begin{equation} \label{eq:l}
	\ell < 3 + 2s.
	\end{equation}
Since $p^s = \ell_p \leqslant \ell$, we have $p^s < 3 + 2s$.  
Therefore either $s=0$ or $(p,s)=(3,1)$.  

\emph{Subcase 2.1.} Suppose that $s = 0$.  By Claim 2 and  \eqref{eq:p-part} we have $k\leqslant m_p = p^r = (q^c-1)_p$ and since $c>1$, $p \nmid q-1$, so $k\leqslant (q^c-1)_p \leqslant \frac{q^c-1}{q-1}$. Now if $\ell \geqslant 2$, we have 
$$k^2 > m = \frac{q^d-1}{q-1}=\frac{q^{\ell c}-1}{q-1} \geqslant \frac{q^{2c}-1}{q-1} = \frac{(q^c+1)(q^c-1)}{q-1} \geqslant k(q^c+1).$$
 Hence $k > q^c$, contradicting $k \leqslant q^c-1$. Thus $\ell=1$, and so $d=c$. Hence $p$ is a primitive prime divisor of $q^d-1$, and so the Sylow $p$-subgroups of $T$ are cyclic (see Remark~\ref{rem:psylcyclic}). Hence  $k=p$. Now by \eqref{eq:p-part} we have $p^r = m_p < k^2 = p^2$. Thus $r=1$.
 
Since $p$ is a primitive prime divisor of $q^d-1$,   $\mathrm{soc}(P^\rho) \cong C_p$ is contained in a Singer cycle subgroup of $T$. Now by \cite[Kapitel II, 7.3]{HuppertI}, $|N_G(\mathrm{soc}(P^\rho))/C_G(\mathrm{soc}(P^\rho))|$ divides $df$ where $q=q_0^f$, $q_0$ a prime,  so we find that $p-1 = |P/\mathrm{soc}(P)|$ divides $df$. 
 In particular, $p-1 \leqslant df$, so $p \leqslant df+1$. Thus  
$$
(df+1)^2 \geqslant p^2 = k^2 > m = q_0^{(d-1)f}+\dots +q_0^f +1 > 2^{(d-1)f}.
$$
The above restrictions give bounds on $d$ and $f$. We first use the condition $(df+1)^2 > 2^{(d-1)f}$ to determine possible values for $d, f$, then for each of these we find the odd primes $p$ such that $p-1$ divides $df$ and $p^2 > 2^{(d-1)f}$. We list the possibilities in Table~\ref{tab: d and f}. 
\begin{table}[h]
\begin{tabular}{ c | c | c | c | c   }
$d$ & $f$ & $df$ & $p$   & $p-1$ \\ 
\hline
2 & 1, 2, 3, 4, 5, 6, 7, 8 &  2, 4, 6, 8, 10, 12, 14, 16 & 3, 5, 7, 11, 13, 17 & 2, 4, 6, 10, 12, 16 \\
3 & 1, 2, 3 & 3, 6, 9 &3, 7 &  2,  6 \\
4 & 1, 2  & 4, 8 &3, 5 &  2, 4 \\
5 & 1 & 5 & -- & -- \\
6 & 1 & 6 & 7 & 6 
\end{tabular}
\caption{Possibilities for $p$ given by $d$ and $f$}
\label{tab: d and f}
\end{table}

In particular, $p\in \{3,5, 7,11,13,17\}$. Note $k=p$ and $n<k$. Then $p=3$ implies that $n=2$ and hence $m=6$. The only solution to $m=(q^d-1)/(q-1)=6$ is given by $d=2$, $f=1$ and $q=5$. Hence $T=\PSL(2,5)$ as in row 5 of Table~\ref{table:2transAS-ii}. If $p=5$, then $n\in \{2,3,4\}$ and $m=kn \in \{10,15,20\}$.  From Table~\ref{tab: d and f} we have $d \in \{2,4\}$. If $d=4$, then the only solution is $q=2$, $T=\PSL(4,2)$ and $n=3$, as in row 13 of Table~\ref{table:2transAS-ii}. For $d=2$, we have $m=q+1$, and  hence $q+1 \in \{10,15,20\}$. If $q=19$, then $f=1$ and $p-1=4$ does not divide $df=2$, a contradiction. If $q=9$, then $n=2$ and $T=\PSL(2,9)$ as in row 7 of Table~\ref{table:2transAS-ii}.

For $p\in \{7, 11, 13, 17\}$, using the bound $p^2 > q_0^{(d-1)f}$,  the possible values for $q_0$ and $m$ informed  by the possible values of $p$ and the pair $(d,f)$ are given in Table~\ref{tab: p q}. If a case leads to a possibility, we give a  reference to Table~\ref{table:2transAS-ii} , and otherwise an explanation of the contradiction that arises.
\begin{table}[h]
\begin{tabular}{c | c | c  | c | l  }
$p$  & $q_0$  & $(d,f)$ & $m$ & notes \\
\hline
7 & 2 &  (2,3)  & 9  &  a contradiction to $p \mid m$ \\
  &    &   (2,6)  &   65 &  a contradiction to $p \mid m$ \\
  &   & (3,2)  &  21 & row 12 of Table~\ref{table:2transAS-ii} \\
  &    &   (6,1)  & 63 & a contradiction to $n  < p$ \\
  & 3   & (2,3) & 28 &  row 9 of Table~\ref{table:2transAS-ii} \\
11& $ 2$  & (2,5) & 33 & row 10 of Table~\ref{table:2transAS-ii} \\
13 & 2 & (2,6) & 65 & row 11 of Table~\ref{table:2transAS-ii}  \\
17 & 2 & (2,8)  & 257  & a contradiction to $p \mid m$
\end{tabular}\caption{Possibilities for $m$ and $n$} \label{tab: p q}
\end{table}

\emph{Subcase 2.2.} Finally suppose that $(p,s) = (3, 1)$. Then  $\ell < 5$ by \eqref{eq:l}; since $\ell_p = \ell_3 = 3^s = 3$ we have $\ell = 3$ and $d=3c$.  Now a Sylow 3-subgroup of $\GL(d,q)=\GL(3c,q)$ is contained in   the stabiliser of a direct sum decomposition of the underlying vector space into three $c$-dimensional subspaces. This stabiliser is isomorphic to $\GL(c,q) \wr \sym(3)$, and since $p=3$ is a primitive prime divisor of $q^c-1$ with $(q^c-1)_3=3^r$, a Sylow $3$-subgroup is isomorphic to $C_{3^r} \wr C_3$. Note that the elementary abelian subgroups of $C_{3^r} \wr C_3$ have order at most $3^3$. Since $c>1$, we can choose an elementary abelian preimage of $C_3^e \cong \soc(P^\rho)$ in $\GL(d,q)$. Hence $e \leqslant 3$ and $k\in \{3,9,27\}$.  Since $k^2>m$, if $c\geqslant 3$, then  $m \geqslant 2^9-1$, but $2^9-1 > 27^2$, a contradiction. Hence $c=2$ (since $c>1$) and  $d=3c=6$. Now since $p=3$ divides $q^c-1$, we have $q=2$ or $q \geqslant 5$. The latter case contradicts $m < k^2$, hence $q=2$. Hence $T=\PSL(6,2)$ and $kn=63$, which forces  $k = 9$, $n = 7$, as in row 15 of Table~\ref{table:2transAS-ii}.
\end{proof}

\begin{thm}[Theorem~\ref{thm:2trans}(3)]
\label{thm:finish 2-trans}
Suppose that $2 \leqslant n < k$ and $k=p^e$ for some prime $p$. Let $P$ be a $2$-transitive affine group of degree $k$ and let $G=\sh(P,n)$. If $n=p^f$ for some integer $f$, then $G$ is an affine group. Otherwise,   $ \alt(kn) \leqslant G $.
\end{thm}
\begin{proof}
If $n=p^f$ for an integer $f$, then Proposition~\ref{prop:affine} shows that $G$ is affine. Thus we may assume $n\neq p^f$ for any $f$. Now Lemma~\ref{lem:2trans-aff} shows that $G$ is 2-transitive, and since the degree of $G$ is $kn$, which is not a prime power, $G$ cannot be affine. Thus Burnside's Theorem \cite[Theorem 4.3]{cameronpermgroups} shows that $G$ is almost simple. Lemma~\ref{lem:must be in table} shows that   $T:=\mathrm{soc}(G)$ appears in Table~\ref{table:2transAS}. If $T \neq \alt(kn)$ then by Lemmas~\ref{lem:2-trans Sp}, \ref{lem:2-trans sz}, \ref{lemma 2-trans Ree PSU} and \ref{lem: 2-trans PSL}, $T$, $k$ and $n$ must appear in Table~\ref{table:2transAS-ii}. For each possible $T$ and value of $k$ and $n$, we obtain a contradiction using {\sc Magma}. Hence $T=\alt(kn)$ as required.
\end{proof}

\section{Cascading shuffle groups}
\label{sec:cascading}

Throughout this section we assume that $k = 2^e$ for some positive integer $e$ and that $n \geqslant 2$.
For each $a \in [k] = \{0, 1, \ldots, 2^e-1\}$, write $a$ in its base-$2$ expansion,
	\[ 
	a = \sum_{r=0}^{e-1} a_r 2^{r}, 
	\]
with each $a_i\in\{0,1\}$. For each $s \in [e]$, let $v_s$ be the permutation on $[k]$ defined by
	\begin{equation} \label{eq:v-action}
	a^{v_s} = \sum_{r=0}^{e-1} \overline{a_r}\, 2^{r}, \ \text{where} \ 
	\overline{a_r} = 
	\begin{cases} a_r &\text{if } r \neq s, \\ 1-a_s &\text{if } r = s. \end{cases}
	\end{equation}
Clearly $v_s$ has order $2$, and $v_r$, $v_s$ commute for all $r,s$. Hence 
	\begin{equation} \label{eq:V}
	V_{e} = \langle v_0, \ldots, v_{e-1} \rangle
	\end{equation}
is an elementary abelian subgroup of $\sym(k)$ of order $2^e$. We now extend these definitions and notation for arbitrary $2^t\leqslant 2^e$.

For each $t=1, \dots, e$, we can divide the cards in $ [2^e n]$ into $2^t$ piles with $2^{e-t}n$ cards in each pile. 
We define $V_{t}$ as in \eqref{eq:V} with $t, 2^t$ in place of $e, k$, namely, for $s\in [t]$, we define $v_{t,s}$ on 
$a=\sum_{r=0}^{t-1}a_r 2^{r}\in [2^t]$ by 
\begin{equation} \label{eq:vt-action}
	a^{v_{t,s}} = \sum_{r=0}^{t-1} \overline{a_r}\, 2^{r}, \ \text{where} \ 
	\overline{a_r} = 
	\begin{cases} a_r &\text{if } r \neq s, \\ 1-a_s &\text{if } r = s. \end{cases}
	\end{equation}
and then
\[
V_t = \langle v_{t,0}, \dots, v_{t,t-1}\rangle \leqslant\sym(2^t)
\]
is elementary abelian of order $2^t$, and the corresponding shuffle group $G_t$ is defined as
	\begin{equation} \label{eq:G_t}
	G_t = \sh\big(V_{t},2^{e-t}n\big).
	\end{equation}
Note that $G_t \leqslant \sym(2^e n)$. Thus the $e$ shuffle groups $G_1, \ldots, G_e$ all act on the same set $[2^e n ]$. Also $G_1 = \sh\big(V_1,2^{e-1}n\big) = \sh\big(\sym(2),2^{e-1}n\big)$, and so $G_1$ is known and is  described by Theorem~\ref{dgk}.

Let $\sigma_t$ denote the shuffle in $G_t$ defined as in \eqref{eq: s} with $2^t, 2^{e-t}n$ in place of $k, n$ respectively.  Let $\sigma = \sigma_1$. We describe in Lemma~\ref{lem:power of s} the relationship between $\sigma$ and the shuffle 
$\sigma_t$ for an arbitrary $t$. 
To describe the other generators for $G_t$ as in Definition~\ref{def:sh}, we need to consider the analogue of the embedding $\rho: \sym(k)\rightarrow 
\sym(kn)$ defined in \eqref{eq: rho}.  Since this map is analogous to $\rho$, we denote it by $\rho_t: \sym(2^t)\rightarrow 
\sym(2^en)$ (so $\rho_e=\rho$), and define for each $\tau\in\sym(2^t)$ the permutation $\rho_{t,\tau}\in\sym(2^en)$ as follows (similarly to \eqref{eq: rho}).
\begin{equation}\label{eq:rhot}
 (a_0(2^{e-t}n)+b_0)^{\rho_{t,\tau}} = a_0^\tau (2^{e-t}n) + b_0     \quad \mbox{for $a_0\in [2^t]$ and $b_0\in [2^{e-t}n]$.}
\end{equation}  
We show in Lemma~\ref{lem:power of s} that the subgroup $V_t$ is mapped by $\rho_t$ to a subgroup of $(V_e)\rho\leqslant G_e$, where $\rho=\rho_e$.

\begin{lem} \label{lem:power of s}
Assume that $k = 2^e$ for some positive integer $e$, and let $t \in \{1, \ldots, e\}$. Then for
$G_t$ the shuffle group as in \eqref{eq:G_t} we have:
\begin{enumerate}

\item[(a)] The shuffle $\sigma_t$ satisfies $\sigma_t = \sigma^t$; 

\item[(b)] for $s\in [t]$, the image $(v_{t,s})\rho_{t}=(v_{e-t+s})\rho$, and so 
$(V_t)\rho_t = (\langle v_{e-t},\dots, v_{e-1}\rangle)\rho \leqslant (V_e)\rho$;

\item[(c)] for $1\leqslant s\leqslant e-1$, $\rho_{v_{s-1}}\sigma = \sigma\rho_{v_s}$;

\item[(d)] the elements $x_r:=(v_{r})\rho$,  for $0\leqslant r\leqslant e-1$, are such that $x_r^\sigma = x_{r+1}$ for $r\leqslant e-2$, 
$(V_t)\rho_t =\langle x_{e-t},\dots,x_{e-1}\rangle$ for $t\leqslant e$, and $((V_t)\rho_t)^\sigma >  (V_{t-1})\rho_{t-1}$ for $t>1$.
\end{enumerate} 
\end{lem}

\begin{proof}
(a) By Lemma \ref{lem:shuffle acts}, $\sigma$ fixes $0$ and $2^e n - 1$, and $i^\sigma = 2i \pmod{2^e n - 1}$ for all other $i \in \big[2^e n\big]$. Likewise $\sigma_t$ fixes $0$ and $2^e n - 1$, and  $i^{\sigma_t} = 2^t i \pmod{2^e n - 1} = i^{\sigma^t}$ for all other $i \in \big[2^e n\big]$. Therefore $\sigma_t = \sigma^t$.

(b) Let $x\in [kn]$.  We may write $x=an+b = a_0(2^{e-t}n)+b_0$ for unique $a\in [2^e], b\in [n], a_0\in[2^t], b_0\in[2^{e-t}n]$. Note that 
$b_0=b+b_1n$ for (a unique) $b_1\in[2^{e-t}]$, and $a=a_02^{e-t}+b_1$.  Write the base-$2$ expansions of $a_0$ and $b_1$ as 
\[
a_0=\sum_{r=0}^{t-1}a_{0,r}2^{r}\quad \mbox{and}\quad b_1=\sum_{r=0}^{e-t-1}b_{1,r}2^{r},
\] 
and note that
\[
\sum_{r=e-t}^{e-1}a_{0,r-e+t}2^{r} + \sum_{r=0}^{e-t-1}b_{1,r}2^{r} = \sum_{r=0}^{t-1}a_{0,r}2^{e-t+r} + \sum_{r=0}^{e-t-1}b_{1,r}2^{r}=a_02^{e-t}+b_1
\]
is the base-$2$ expansion of $a=a_02^{e-t}+b_1$. 
By \eqref{eq:vt-action}, setting $\overline{a_{0,r}} = a_{0,r}$ if $r \neq s$, and  $\overline{a_{0,r}} = 1-a_{0,r}$ if $r=s$, we have
\[
(a_0^{v_{t,s}})2^{e-t}+b_1 = \sum_{r=0}^{t-1}\overline{a_{0,r}}2^{e-t+r}\ +\ \sum_{r=0}^{e-t-1}b_{1,r}2^{r},
\]
and since $s\leqslant t-1$ and $b_1\leqslant 2^{e-t}$, it follows from \eqref{eq:v-action} that this expression is the image of $a$ under the action of $v_{e,e-t+s}$. Hence, by \eqref{eq:rhot}, 
\begin{align*}
x^{\rho_{t,v_{t,s}}} &= a_0^{v_{t,s}}(2^{e-t}n)+b_0\\
&= (a_0^{v_{t,s}}2^{e-t}+b_1)n+b\\
&= a^{v_{e,e-t+s}}n+b\\
&= x^{\rho_{v_{e-t+s}}}.
\end{align*}
We conclude that $(v_{t,s})\rho_t=\rho_{t,v_{t,s}}= \rho_{v_{e-t+s}}=(v_{e-t+s})\rho$, and part (b) follows.

(c) Let $1\leqslant s\leqslant e-1$, and let $x=an+b\in [2^en]$ with $a\in[2^e]$ and $b\in [n]$. We will show that the images of $x$ under 
$\rho_{v_{s-1}}\sigma$  and $\sigma\rho_{v_s}$ are equal. Write the base-$2$ expansion of $a$ as 
\[
a=\sum_{r=0}^{e-1}a_{r}2^{r},
\] 
and define 
\[
b_1=\begin{cases}
	1 & \mbox{if either $b\geqslant n/2$,	}	\mbox{or $b=(n-1)/2$ and $a_{e-1}=1$,} \\
	0 & \mbox{if either $b\leqslant (n-2)/2$,}\ \mbox{or $b=(n-1)/2$ and $a_{e-1}=0$.} \\  
\end{cases} 
\]
By Lemma \ref{lem:shuffle acts}, $\sigma$ fixes $2^e n - 1$, and $i^\sigma = 2i \pmod{2^e n - 1}$ for all other $i \in \big[2^e n\big]$. 
We claim that, for $x\ne 2^en-1$, 
\begin{equation}\label{xs}
x^\sigma = 2x = (\sum_{r=1}^{e-1}a_{r-1}2^{r} + b_1) n + (2b-b_1n+a_{e-1}) \pmod{2^en-1}
\end{equation}
with $\sum_{r=1}^{e-1}a_{r-1}2^{r} + b_1$ the base-$2$ expansion of an element in $[2^e]$, and $2b-b_1n+a_{e-1}\in [n]$. Note that $2an=
(\sum_{r=0}^{e-1}a_{r}2^{r+1})n= a_{e-1} + (\sum_{r=1}^{e-1}a_{r-1}2^{r})n  \pmod{2^en-1}$, and hence equality holds in \eqref{xs},
and  since $b_1\in\{0,1\}$, $\sum_{r=1}^{e-1}a_{r-1}2^{r} + b_1$ is the base-$2$ expansion of an element in $[2^e]$. If $b\geqslant n/2$ then 
$0\leqslant 2b-b_1n+a_{e-1} \leqslant 2(n-1)-n+1=n-1$; if $b\leqslant (n-2)/2$ then $0\leqslant 2b-b_1n+a_{e-1} \leqslant (n-2)+1=n-1$; and finally if $b=(n-1)/2$, then 
$(a_{e-1},b_1)$ is either $(1,1)$ or $(0,0)$ and $2b-b_1n+a_{e-1}$ is $0$ or $n-1$ respectively. This proves the claim. 
Since $s\geqslant1$, this implies that 
\begin{equation}\label{xsr}
x^{\sigma(v_s)\rho}  = (\sum_{1\leqslant r\leqslant e-1, r\ne s}a_{r-1}2^{r} + (1-a_{s-1})2^s + b_1) n + (2b-b_1n+a_{e-1}).
\end{equation}
Also, provided that $x^{(v_{s-1})\rho}\ne 2^en-1$, we may apply \eqref{eq:v-action} and \eqref{xs} to obtain 
\begin{align*}
x^{(v_{s-1})\rho\sigma} &= \left(    (\sum_{0\leqslant r\leqslant e-1, r\ne s-1}a_{r}2^{r} + (1-a_{s-1})2^{s-1}) n+b\right)^\sigma \\
		&=   (\sum_{1\leqslant r\leqslant e-1, r\ne s-1}a_{r-1}2^{r} + (1-a_{s-1})2^{s}+b_1) n +  (2b-b_1n+a_{e-1})
		= x^{\sigma(v_s)\rho}.
\end{align*}
On the other hand, if $x^{(v_{s-1})\rho}= 2^en-1$, then $x= (\sum_{0\leqslant r\leqslant e-1, r\ne s-1}2^r )n+(n-1)$, $b=n-1, b_1=1, a_{e-1}=1$, and by \eqref{xsr},
\begin{align*}
x^{\sigma(v_s)\rho} &=\left((\sum_{1\leqslant r\leqslant e-1, r\ne s}2^{r}) + 2^s + b_1\right) n + (2b-b_1n+a_{e-1})\\
	&=  (\sum_{0\leqslant r\leqslant e-1}2^{r}) n + (n-1) = 2^en-1 = x^{(v_{s-1})\rho}.
\end{align*}
It remains to consider actions on $x=2^en-1 = 
(\sum_{r=0}^{e-1}2^r)n+(n-1)$. Here we have
\[
x^{\sigma(v_s)\rho}  = x^{(v_s)\rho} = (\sum_{0\leqslant r\leqslant e-1, r\ne s}2^{r}) n + (n-1),
\]
while, using the fact that $x^{(v_{s-1})\rho}\ne 2^en-1$, we have
\begin{align*}
x^{(v_{s-1})\rho\sigma} &=  2\left(   (\sum_{0\leqslant r\leqslant e-1}2^{r} - 2^{s-1}) n + n-1\right)\pmod{2^en-1} \\
		&=   (\sum_{0\leqslant r\leqslant e-1}2^{r+1}  -2^s) n + 2(n-1)\\
		&= (\sum_{0\leqslant r\leqslant e-1, r\ne s}2^{r}) n + n-1\pmod{2^en-1} = x^{\sigma(v_s)\rho} .
\end{align*}
This completes the proof of part (c). Part (d) follows immediately from parts (b) and (c).
\end{proof}

By Lemma~\ref{lem:power of s}(b)  we have the following subgroup chain in $G_e$: 
	\[ 
	V_1\rho_1  \leqslant \dots \leqslant V_{t}\rho_t  \leqslant \dots \leqslant  V_{e}\rho. \]

We now consider how the subgroups $G_1$, \dots, $G_e$ are related.

\begin{lem}
\label{lem:containment}
For $1 \leqslant t \leqslant e$ we have
$$G_t = \langle x_{e-t},\dots,x_{e-1}, \sigma^t \rangle.
$$
Moreover, $G_t\leqslant G_1$, and if $t<e$ then $G_t$ normalises $G_{t+1}$. 
\end{lem}
\begin{proof}
By definition, $G=\sh(V_t,n) = \langle (V_t)\rho_t, \sigma_t \rangle$. By Lemma~\ref{lem:power of s} part (a) we have $\sigma_t=\sigma^t$ and by part (c) we have $(V_t)\rho = \langle x_{e-t},\dots,x_{e-1} \rangle$. This gives the first part of the lemma. For the next assertion, we may assume $e \geqslant 2$. 
By Lemma~\ref{lem:power of s}(d), for $1 \leqslant i \leqslant t$ we have
$$x_{e-i} = (x_{e-1})^{\sigma^{-i+1}} \in G_1$$
and clearly $\sigma^t \in G_1$, so the inclusion $G_t\leqslant G_1$ follows. Finally we check that $G_t$ normalises $G_{t+1}$. 
Since $G_{t+1}$ contains $x_{e-t},\dots,x_{e-1}$,  we only need  to check that $ \sigma^t $ normalises $G_{t+1}$. 
Now $\sigma^t$ centralises $\sigma^{t+1}$ and, by Lemma~\ref{lem:power of s}(d),  
$$
(x_{e-t-1})^{\sigma^t} =  x_{e-1} \in G_{t+1}.
$$
 For $1 \leqslant i \leqslant t$, we   use Lemma~\ref{lem:power of s}(d) to write $x_{e-i} = (x_{e-i-1})^\sigma$, and then we have
$$
(x_{e-i})^{\sigma^t} = ((x_{e-i-1})^\sigma)^{\sigma^t} = (x_{e-i-1})^{\sigma^{t+1}} \in G_{t+1}
 $$
which completes the proof.
\end{proof}

We now prove Theorem~\ref{thm:k=2^e} which we restate for convenience. We use notation from Theorem~\ref{dgk}.
 
\setcounter{section}{1}
\setcounter{thm}{8}

\begin{thm}
\label{lem:identify Gr}
Suppose that $n$ is not a power of $2$ and $k=2^e$ with $e\geqslant 2$. The following hold:
\begin{enumerate}
\item if $(k,n)=(4,3)$ then $G_1 = C_2^6\rtimes \sym(5)$ and $G_2 = C_2^5 \rtimes \sym(5)$;
\item if $(k,n)=(4,6)$ then $G_1 =G_2$ and $G_1 \cong C_2^{11} \rtimes M_{12}$;
\item if $(k,n)=(8,3)$ then $G_1 =G_2= G_3$ and $G_1 \cong C_2^{11} \rtimes M_{12}$;
\item if $k=4$ and $n\geqslant 5$ is odd then $G_1 = C_2 \wr \sym(2n)$ and $G_2 = \ker(\mathrm{sgn})$;
\item in all other cases, $G_1 = G_2 = \dots = G_e$ and $G_1 =\ker(\mathrm{sgn}) \cap \ker( \overline{\mathrm{sgn}})$.
\end{enumerate}
%
\end{thm}
\begin{proof}
The structure of $G_1= \sh(\sym(2),2^{e-1}n)$ is given in Theorem~\ref{dgk}.  Note that since $e\geqslant 2$, $2^{e-1}n$ is even. Now, since $n$ is not a power of $2$, $G_1$ appears in row two, four, six or seven of Table~\ref{tab:the shuffle group}. 
First we deal with the exceptional cases.

The exceptional case in row six  is when $2^{e-1}n=6$. Here $2^e=4$, $2^en=12$, and $G_1 = \sh(\sym(2),6) \cong 2^6 \rtimes \sym(5)$. By {\sc Magma}, $\sh(V_4,3)$ has index two in $G_1$, and $G_2 \cong 2^5 \rtimes \sym(5)$.

The exceptional case in row seven is when $2^{e-1}n = 12$, so $(2^e,n)=(4,6)$ or $(8,3)$, and $G_1= 2^{11} \rtimes M_{12}$, where $M_{12}$ is  one of  Mathieu's nonabelian simple groups. 
By Lemma~\ref{lem:containment}, $G_2$  is a normal subgroup of $G_1$.  Since $\sigma^2$ does not centralise $\langle x_{e-2}, x_{e-1}\rangle$,   $G_2$ is nonabelian, and hence $G_2=G_1$. Similarly, if $e=3$, then $G_3$ is a nonabelian normal subgroup of $G_2 = G_1$, so  $G_1 = G_2 = G_3 $.

Suppose now that $2^{e-1}n \neq 6$ and $2^{e-1}n \equiv 2 \pmod 4$ (as in row four of Table~\ref{tab:the shuffle group}). Then $2^e=4$, $n$ is odd, and $G_1 = B_{2n} = 2^{2n}.\sym(2n) \cong C_2 \wr \sym(2n)$. By Lemma~\ref{lem:containment}, $G_2$ is a normal subgroup of $G_1$, and as above, $G_2$ is nonabelian. The proper nonabelian normal subgroups of $G_1$ all contain the derived subgroup of $G_1$, namely $[G_1,G_1] \sim2^{2n-1}.\alt(2n)$. Hence $G_2$ contains $[G_1,G_1]$ and so has index one, two or four in $G_1$. By Corollary~\ref{cor:parity}(3)  we have that $G_2 = \sh(V_4,n)  \leqslant \alt(4n)$, and so $G_2$ is contained in $\ker(\mathrm{sgn})$ which has index two in $B_{2n}$. Note that the unique index four subgroup of $G_1$ projects to $\alt(2n)$ in the quotient $G_1 / \mathrm O_2(G_1) \cong \sym(2n)$. To determine the index of $G_2$, we consult \cite[Table III]{DGK} - and note that, unfortunately, $2n$ corresponds to the the integer $n$ in \cite[Table III]{DGK} (because they are investigating $\sh(\sym(2),n)$ and we are investigating $G_1 = \sh(\sym(2),2n)$). We see that $\sigma$ (there denoted $I$) projects to an odd permutation in $G_1/\mathrm O_2(G_1)$ (witnessed in the table in column $2 \pmod 4$ by $\mathrm{sgn}(\overline{g})=-1$ for $g=I$) and that the element there denoted $O$ which is equal to $ x_{e-1} \sigma$ projects to an even permutation (witnessed in the table in column $2\pmod 4$ by $\mathrm{sgn}(\overline{g})=1$ for $g=O$). Thus $\mathrm{sgn}(x_{e-1}) = -1$ and so  $x_{e-1}$ also projects to an odd permutation in $G_1/\mathrm O_2(G_1)$, and since  $x_{e-1}$ lies in $G_2$, we have that $G_2$ projects to $\sym(2n)$. Hence $G_2 / \mathrm O_2(G_2) \cong (G_2\mathrm O_2(G_1))/\mathrm O_2(G_1)  \cong \sym(2n)$ and so $G_2$ has index exactly two in $G_1$ and we have $G_2 = \ker(\mathrm{sgn})$.

Now consider the case that $2^{e-1}n \equiv 0 \pmod 4$ and $n$ is not a power of two (as in row two of Table~\ref{tab:the shuffle group}). From Table~\ref{tab:the shuffle group} we have  $G_1 = 2^{2^{e-1}n-1}.\alt(2^{e-1}n)$. Since $2^{e-1}n$ is even, we see that $G_1$ has exactly two proper non-trivial normal subgroups, both are abelian and have order $2$ and $2^{2^{e-1}n-1}$ respectively. By Lemma~\ref{lem:containment},  $G_2 $ is a normal subgroup of $G_1$. Since $G_2$ is nonabelian, we have $G_2=G_1$. If $e \geqslant 3$, then by Lemma~\ref{lem:containment}, $G_3\leqslant G_1$, and since $G_1=G_2$, we have $G_3 \leqslant G_2$.  Then, again by Lemma~\ref{lem:containment}, $G_3$ is a normal subgroup of $G_2=G_1$, and since $G_3$ is nonabelian, we conclude that $G_3=G_2=G_1$. Continuing in this fashion, we have $G_1=G_2=\dots=G_e$. This completes the proof of the theorem.
%
\end{proof}
 
 We are now ready to prove Corollary~\ref{cor:k=2^e} which we also restate for convenience.
 
\begin{cor} 
 Suppose that $n$ is not a power of 2 and that $k = 2^e$ for some $e\geqslant 2$. Then $\sh(\sym(k),n)$ is the  alternating group  or the symmetric group if $n$ is even or odd, respectively.
\end{cor}
\begin{proof}
Set $G=\sh(\sym(k),n)$ and $G_e=\sh(V_{2^e},n)$. Note that the restrictions on $e$ and $n$ imply that $2^en\geqslant 12$. The structure of $G_e$ is given by  Theorem~\ref{lem:identify Gr}. For $(k,n)=(4,3)$, $(4,6)$ or $(8,3)$, we verify that the result holds using {\sc Magma}. Thus we may assume one of case (4) or (5) of Theorem~\ref{lem:identify Gr} holds. In particular, we have that $G_e$ contains the derived subgroup  $L=\ker(\mathrm{sgn}) \cap \ker(\overline{\mathrm{sgn}})=2^{2^{e-1}n-1}.\alt(2^{e-1}n)$ of $G_1$. Note that $L$ is a transitive imprimitive subgroup of $\sym(2^en)$ preserving a partition with $2^{e-1}n$ parts of size $2$. 
Suppose that $\alt(2^en)G \neq G$   and let $M$ be a maximal subgroup of $\alt(2^en)G$ containing $G$. Since $M$ contains $G$, we have that $M$ contains $L$. Now  $L$ contains elements of cycle type $5^2$ and so, since we have dealt with the case $2^en = 12$ above, applying   \cite[Theorem 13.10]{wielandt} with $q=2$ and $p=5$ yields  that $M$  is imprimitive. On the other hand, Theorem~\ref{thm:prim} implies that $G$ is primitive, a contradiction. Hence $\alt(2^en)G = G$ and so  $G$ contains $\alt(2^en)$. Since $e \geqslant 2$, $k\equiv 0 \pmod{4}$, so Corollary~\ref{cor:parity} shows that $G \leqslant \alt(2^en)$ if and only if $n$ is even.
\end{proof}

\end{document}